\documentclass[10pt,a4paper]{amsart}
\usepackage[T1]{fontenc}
\usepackage{lmodern,amsmath,amssymb,mathtools,booktabs,array}
\usepackage[a4paper,textwidth=6.95in,textheight=10.32in,centering]{geometry}
\usepackage{microtype}
\usepackage{xcolor}
\usepackage[colorlinks=true,linkcolor=black,citecolor=black,urlcolor=blue!45!black]{hyperref}
\hypersetup{pdftitle={Harmonic weights in inverse central binomial series},pdfauthor={K. Srinivasa Raghava}}
\newtheorem{theorem}{Theorem}[section]
\newtheorem{lemma}[theorem]{Lemma}
\newtheorem{proposition}[theorem]{Proposition}
\newtheorem{corollary}[theorem]{Corollary}
\theoremstyle{remark}
\newtheorem{remark}[theorem]{Remark}

\numberwithin{equation}{section}
\DeclareMathOperator{\Cl}{Cl}
\allowdisplaybreaks[1]
\title[Harmonic weights in inverse central binomial series]{Harmonic weights in inverse central binomial series}
\author{K. Srinivasa Raghava}
\address{Pie Mathematics Association}
\email{srinivasaraghavak@gmail.com}
\date{November 2026}
\subjclass[2020]{Primary 11M06; Secondary 11B65, 33B30, 40A25, 44A60}
\keywords{Inverse central binomial sum, harmonic number, Stieltjes function, Fourier sine series, Clausen value, log-sine integral}
\begin{document}

\begin{abstract}
We compare the masses at the poles of an inverse central binomial series
with harmonic weights \(H_{pn}\) and \(H_{qn}\), for integers \(p<q\) with \(p>0\). Each mass increases
by an amount tending to \(2\log(q/p)\), uniformly in the angle. The increase
is above the limit at every odd pole and below it at every even pole.
For real dilations we determine the exact range of this sign law.
For \(p\ge1\) it holds exactly when \(q\ge2\). For \(0<p<1\) it holds
above a threshold \(Q(p)\). It holds at the threshold exactly on one
interval \(0<p<p_*\). At \(p_*\) the first pole at limiting angle zero
ties with the second at angle \(\pi\). Only the first pole determines
the threshold below \(p_*\). The second pole then determines it at
\(\pi\), moves to an interior angle, and gives way to the fourth pole
at \(\pi\). As \(p\) tends up
to one, every sufficiently large even pole takes its turn in increasing
order. Its determining angle equals \(\pi\) on the first part of its
interval and lies strictly below \(\pi\) on the second.
The first correction has order \(k^{-2/q}\) for \(q>2\), and order
\(k^{-2}\) for the integer pair \((p,q)=(1,2)\). The comparison gives the sharp
subtraction criterion for the Stieltjes property. P\'olya's formula gives
the partial fractions. For \(H_n\) and \(H_{2n}\) we evaluate the masses
at rational angles. We give angle identities at every odd weight from
three onwards. At \(\pi/3\) we classify the cancellations of Dirichlet
\(L\)-terms and give a second proof of the evaluations using Bailey's
cubic transformation.
\end{abstract}
\maketitle
\abovedisplayskip=3.1pt plus 2pt minus 1pt \belowdisplayskip=\abovedisplayskip
\belowdisplayshortskip=1.5pt plus 1pt \jot=2pt
\lineskiplimit=0pt \normallineskiplimit=0pt

\section{Introduction}
The harmonic numbers \(H_n=\sum_{j=1}^n 1/j\) satisfy
\(H_{qn}-H_{pn}\to\log(q/p)\) for positive integers \(p<q\).
We study the same change of weight in a series whose partial fraction
coefficients depend on an angle.

Fix \(0<\phi\le\pi\), write \(t=2\sin(\phi/2)\), and set
\begin{equation}\label{intro:w}
 w_n(u)=\frac{1}{n^2\binom{2n}{n}}
        \prod_{j=1}^{n-1}\left(1-\frac{u^2}{j^2}\right).
\end{equation}
For a positive integer \(q\), define
\begin{equation}\label{intro:normal}
 X_q(u)=\sum_{n\ge1}t^{2n}w_n(u)H_{qn},\qquad
 T_q(s)=\frac{\phi\sqrt{s}}{\sinh(\phi\sqrt{s})}\,X_q(i\sqrt{s}).
\end{equation}
Put \(T_q(0)=X_q(0)\) and \(\mu_k=k\pi/\phi\). We call the coefficient
of \(1/(s+\mu_k^2)\) its \emph{mass} and denote it by \(\rho_k(q)\).

\begin{theorem}[Comparison of harmonic weights]\label{spec:main}
For every positive integer \(q\) and \(s\ge0\),
\begin{equation}\label{spec:pf}
 T_q(s)=\sum_{k\ge1}\frac{\rho_k(q)}{s+\mu_k^2},\qquad
 \rho_k(q)=2(-1)^{k+1}\mu_k^2X_q(\mu_k).
\end{equation}
The masses are positive and the series converges absolutely. For fixed positive
integers \(p<q\), every mass increases. The increase exceeds \(2\log(q/p)\)
for every odd \(k\) and is smaller for every even \(k\).
For real \(c\), \(T_q-cT_p\) is Stieltjes exactly when \(c\le1\).
Put \(a=2/q\) and \(h=\min(2a,2/p)>a\). As \(k\to\infty\), uniformly for
\(0<\phi\le\pi\),
\begin{equation}\label{spec:gap}
 \rho_k(q)-\rho_k(p)-2\log(q/p)=
 \begin{cases}
 \displaystyle\frac{2(-1)^{k+1}\Gamma(1+a)\cos(\pi a/2)}{(t\mu_k)^a}
       +O(k^{-h}), & q>2,\\[7pt]
 \displaystyle\frac{1+\sin^2(\phi/2)-4(-1)^k}{2(t\mu_k)^2}
       +O(k^{-4}), & (p,q)=(1,2).
 \end{cases}
\end{equation}
\end{theorem}

The starting point is P\'olya's partial fraction formula for a Fourier sine
transform divided by a sine \cite[eq.~(44)]{Polya}. Its coefficients are samples of the transform.
For increasing kernels their signs also give interlacing of the zeros
\cite{ChoPark,CCP}. Equation~\eqref{spec:pf} is this formula applied to
a harmonic weight. Here the difference of the two harmonic kernels is
itself increasing, so the same sign argument orders the masses at every
pole. Its total change is \(\log(q/p)\), which gives the limiting mass
difference. For \(q>2\) its leading power at zero gives the first correction,
independent of \(p\). For \(q=2\) both ends of the kernel contribute.
Since \(t\mu_k/(k\pi)=\sin(\phi/2)/(\phi/2)\) ranges from \(2/\pi\)
to the limiting value \(1\), the leading coefficient stays away from zero.
The expansion holds even for angles tending to zero or to \(\pi\) at any
rate. To obtain the sign at every index, we compare consecutive harmonic
weights and prove that the derivative of their difference kernel is
decreasing and convex. By
Proposition~\ref{spec:growth}, an individual mass grows like \(2\log k\) at
each fixed interior angle and like \(4\log k\) at \(\pi\). Its growth and
the bounded difference explain the sharp subtraction threshold.
Theorem~\ref{real:theorem} gives the range for real dilations.
Propositions~\ref{real:smallp} and~\ref{real:nearone} identify the two
endpoint regimes. The first pole already drops out for \(2/3\le p<1\)
by Proposition~\ref{real:firstexcluded}.
Corollary~\ref{real:changes} gives the order of all subsequent changes
sufficiently near one.
Theorem~\ref{real:middle} proves that equality is lost at a single
transition, where the first pole ties with the second. Its proof combines
a strict comparison along the first-pole root with certified bounds for
the other indices. Corollaries~\ref{real:secondangle} and
\ref{real:secondfourth} locate a change of angle at the second pole
and a later tie between the second and fourth poles.

For the angle formulas write
\begin{equation}\label{intro:functions}
 N_\phi(u)=\sum_{n\ge1}t^{2n}w_n(u)/n,\qquad A_\phi=X_1,\qquad B_\phi=X_2.
\end{equation}
Define \(e_r(n)\) by \(\prod_{j<n,\,j\ge1}(1+z/j^2)=\sum_{r\ge0}e_r(n)z^r\).
Thus \(e_0=1\) and \(e_1=H_{n-1}^{(2)}=\sum_{j=1}^{n-1}j^{-2}\).
The coefficient of \((-1)^ru^{2r}\) inserts \(e_r(n)\) into each sum.
We call its weight \(2r+3\).

Lehmer and Zucker evaluated inverse central binomial sums with powers of
\(n\) in the denominator \cite{Lehmer,Zucker}. The finite products in
\eqref{intro:w} are familiar from zeta generating functions
\cite{Koecher,Lesch} and powers of the inverse sine \cite{BC}.
Davydychev and Kalmykov give the weight-three angle evaluations and
formulas for the component sums at weight five \cite{DK}.
Chavan and Chavan give elementary proofs of fixed-angle evaluations
involving \(H_n\) and \(H_{2n}\) \cite{Chavan}.
At \(t=1\), Akhilesh's formulas with repeated twos and one index equal to
one or three cover the \(1/n\) and \(H_n\) columns \cite[Theorems 5--7]{Akhilesh}.
For \(H_n\), split off \(1/n\). Multiplying by \(H_{n-1}\) then inserts a
new index or merges it with an existing one. There are related evaluations
by log-sine integrals and multiple polylogarithms \cite{BBK,Gonzalez,KWY,Zhou,SunZhou}.

Here the three angle formulas keep the same parameter \(u\), including in
the \(H_{2n}\) column. At \(\pi/3\), Theorem~\ref{fix:class} determines
exactly which coefficients of the displayed Dirichlet \(L\)-values vanish.
This is a statement about the formulas themselves and requires no
independence assumption on special values. Theorem~\ref{co:pair} treats
the case of one common polynomial. Appendix~\ref{oh:classical} derives the
fixed-angle evaluations independently from classical transformations.

\section{The harmonic kernels and their partial fractions}\label{spec:section}
\subsection{A common integral}
Put \(b=\phi/2\). The starting point is the even-power expansion of the cosine of an arcsine.
\begin{lemma}\label{ang:arcsine}
For \(0\leq\theta\leq\pi/2\) and \(u\in\mathbb C\),
\begin{equation}\label{ang:seed}
 \sum_{n\geq1}(2\sin\theta)^{2n}w_n(u)
       =\frac{1-\cos(2u\theta)}{u^2}.
\end{equation}
The value at \(u=0\) is \(2\theta^2\). For fixed \(0<\phi\leq\pi\), the series defining \(N_\phi\) and \(X_q\), \(q=1,2,\ldots\), are even entire functions of \(u\). Their Taylor coefficients may be taken term by term.
\end{lemma}
\begin{proof}
On \(|u|\leq R\), the absolute value of the finite product in \(w_n(u)\) is bounded by \(\prod_{j\geq1}(1+R^2/j^2)<\infty\). Stirling's formula gives \(4^n/\binom{2n}{n}=O(\sqrt n)\). Thus \(t^{2n}w_n(u)=O_R(n^{-3/2})\), even at \(t=2\). Multiplication by \(H_{qn}=O_q(\log(n+1))\) or by \(1/n\) preserves summability. Local uniform convergence proves the assertions about the series and their coefficients.

For \(|z|<1\), the function \(y(z)=\cos(2u\arcsin z)\) solves
\((1-z^2)y''-zy'+4u^2y=0\), with \(y(0)=1\) and \(y'(0)=0\).
Its even coefficients satisfy
\((2n+2)(2n+1)[z^{2n+2}]y=4(n^2-u^2)[z^{2n}]y\).
The constant coefficient is \(1\). For \(n\geq1\) the recurrence gives
\([z^{2n}]y=-u^2\,4^n w_n(u)\). Substitute \(z=\sin\theta\).
The bound above gives continuity at \(z=1\), and the limit as \(u\to0\) gives the remaining value.
\end{proof}

\begin{lemma}[Integral representation]\label{ang:kernel}
Write \(s=\sin\theta/\sin b\) for \(0<\theta<b\). The kernels for \(N_\phi\) and \(X_q\) are
\[
 K_N(\theta)=-2\log s,\qquad K_q(\theta)=-\log(1-s^{2/q}).
\]
In particular \(K_A=K_1=-\log(1-s^2)\) and \(K_B=K_2=-\log(1-s)\).
If \(X\) denotes any of these functions and \(K_X\) its kernel, then
\begin{equation}\label{ang:int}
 X(u)=\int_0^b \frac{2\sin(2u\theta)}{u}\,K_X(\theta)\,d\theta .
\end{equation}
At zero the first factor in the integral equals \(4\theta\).
\end{lemma}
\begin{proof}
Let \(G(\theta)=(1-\cos(2u\theta))/u^2\), with its limiting value at \(u=0\).
In the harmonic integral \(H_{qn}=\int_0^1(1-y^{qn})/(1-y)\,dy\), define \(\sin\theta_y=\sin b\,y^{q/2}\).
Equation \eqref{ang:seed} gives
\[
 X_q(u)=\int_0^1\frac{G(b)-G(\theta_y)}{1-y}\,dy .
\]
Absolute convergence in Lemma~\ref{ang:arcsine} justifies interchanging sum and integral, since the factor \((1-y^{qn})/(1-y)\) is nonnegative. Under \(y=s^{2/q}\), the measure \(dy/(1-y)\) is \(dK_q(\theta)\). Integrating by parts gives \(\int_0^b G'K_q\). The boundary term at \(b\) vanishes because \(G(b)-G(\theta)=O(b-\theta)\) and \(K_q=O(1+|\log(b-\theta)|)\), also when \(b=\pi/2\). At zero \(K_q\) tends to zero.

For \(N_\phi\), use \(1/n=\int_0^1y^{n-1}\,dy\). With \(\sin\theta_y=\sin b\,\sqrt y\), this gives \(N_\phi(u)=\int_0^1G(\theta_y)\,dy/y\).
Here \(dy/y=-dK_N\). The same integration by parts applies since \(K_N(b)=0\) and \(G(\theta)=O(\theta^2)\) at zero.
\end{proof}

\subsection{The masses and their differences}
A Stieltjes function has the form \(a/s+b+\int_{(0,\infty)}(s+x)^{-1}\,d\nu(x)\),
where \(a,b\ge0\), the measure \(\nu\) is positive and
\(\int(1+x)^{-1}d\nu(x)<\infty\) \cite[Ch.~2]{SSV}.
Here \(a=b=0\).

\begin{proof}[Proof of Theorem~\ref{spec:main}]
Put \(b=\phi/2\), \(r_b(y)=\sin(by)/\sin b\), and
\(g_q(y)=-\log(1-r_b(y)^{2/q})\). Rescale \eqref{ang:int} to get
\[
 X_q(u)=\frac{\phi}{u}\int_0^1\sin(\phi uy)g_q(y)\,dy,\qquad
 T_q(s)=\phi^2\int_0^1\frac{\sinh(zy)}{\sinh z}g_q(y)\,dy,\quad z=\phi\sqrt{s}.
\]
The quotient equals \(y\) at zero. Its \(k\)th Fourier sine coefficient is
\(2(-1)^{k+1}k\pi/(z^2+k^2\pi^2)\), by integration. Both factors in the
second integral are square integrable, since \(g_q\) has only a logarithmic
singularity at one. Parseval's formula gives \eqref{spec:pf}, with absolute
convergence by Cauchy--Schwarz. The mass is
\(2(-1)^{k+1}k\pi\int_0^1g_q(y)\sin(k\pi y)\,dy\).
This is the Fourier sampling formula of P\'olya
\cite[eq.~(3.4)]{ChoPark}. Integrating by parts gives
\begin{equation}\label{spec:positive}
 \rho_k(q)=2\int_0^1 (1-\cos(k\pi v))\,g_q'(1-v)\,dv .
\end{equation}
The boundary term vanishes since \(g_q(0)=0\) and the quadratic zero of
\(1-\cos(k\pi v)\) at \(v=0\) absorbs the logarithmic singularity of
\(g_q(1-v)\). Since \(g_q'>0\) in the open interval, each mass is positive.

For \(p<q\) write the difference \(d_b=g_q-g_p\) as \(d_b(y)=D(r_b(y))\), where
\[
 D(r)=\log\frac{1-r^{2/p}}{1-r^{2/q}},\qquad
 D(0)=0,\qquad D(1)=\log(q/p).
\]
The value at one follows from \(1-r^a\sim a(1-r)\).
The derivative of \(-\log(1-r^a)\)
with respect to \(\log r\) is \(a/(\exp(a|\log r|)-1)\). For \(0<r<1\)
this decreases strictly with \(a>0\), as \(e^x(1-x)<1\) for \(x>0\).
Taking \(a=2/q,2/p\) shows that \(D'>0\) and hence \(d_b'>0\).
Integrating on interior intervals and taking the endpoint limits gives
\(\int_0^1d_b'(y)\,dy=D(1)-D(0)=\log(q/p)\). Subtracting \eqref{spec:positive}
gives strict increase at every pole and the exact identity
\begin{equation}\label{spec:correction}
 \rho_k(q)-\rho_k(p)=2\log(q/p)
     -2\int_0^1 d_b'(1-v)\cos(k\pi v)\,dv.
\end{equation}

Write \(K=k\pi\) and \(A=b/\sin b\), so \(A/K=1/(t\mu_k)\).
With \(y=1-v\), the integral in \eqref{spec:correction} becomes
\((-1)^k\int_0^1d_b'(y)\cos(Ky)\,dy\). Set \(A=1\) and \(r_0(y)=y\)
at \(b=0\). All estimates below hold on the closed interval
\(0\le b\le\pi/2\). Only the kernel needs a value at angle zero.

First take \(q>2\). Subtract the leading power at \(y=0\) by writing
\[
 d_b'(y)=aA^a y^{a-1}+R(y),\qquad
 R^{(j)}(y)=O(y^{h-1-j})\quad(j=0,1,2).
\]
To check the bounds, write \(r_b(y)=Ay\,\sin(by)/(by)\), with the
last factor equal to one at zero. This factor is smooth and differs
from one by \(O(y^2)\), uniformly in \(b\). Expand each logarithm in
\(D\) using \(-\log(1-z)=\sum_{m\ge1}z^m/m\). After \((Ay)^a\)
the remaining powers in \(d_b(y)\) start at \(2a\), \(2/p\), or
\(a+2\). On a sufficiently small fixed interval these series converge
geometrically, even after three applications of \(y\,d/dy\), since
differentiation introduces at most a cubic factor in \(m\).
This gives the three bounds near zero. Away from zero they follow by
smoothness on a closed rectangle. In particular, at \(r=1\) each
quotient \((1-r^c)/(1-r)\), \(c>0\), is smooth and has value \(c\).
Thus the ratio inside \(D\) stays smooth even when \(b=\pi/2\).

The integral of \(R(y)\cos(Ky)\) up to \(1/K\) is \(O(K^{-h})\).
Integrate twice by parts from \(1/K\) to one. The terms at \(1/K\)
are \(O(K^{-h})\), and the term at one is \(O(K^{-2})\) since
\(\sin K=0\). The remaining integral is
\(O(K^{-2}\int_{1/K}^1 y^{h-3}\,dy)=O(K^{-h})\), since \(h<2\).
For the leading power the classical cosine integral gives
\[
 \int_0^1 y^{a-1}\cos(Ky)\,dy
 =K^{-a}\Gamma(a)\cos(\pi a/2)+O(K^{-2}).
\]
Indeed, \(0<a<1\), so integration to infinity gives the stated main
term \cite[eq.~(5.9.6)]{DLMF}. Two integrations by parts on \([1,\infty)\)
give the error because \(\sin K=0\). Substitution in
\eqref{spec:correction} proves the first line of \eqref{spec:gap}.

For \(p=1,\ q=2\) the kernel simplifies to \(d_b(y)=\log(1+r_b(y))\).
Its first five derivatives are bounded uniformly, since \(1+r_b\ge1\).
Direct differentiation gives \(d_b''(0)=-A^2\) and
\(d_b''(1)=-(A^2+b^2)/4\). Two integrations by parts give the endpoint
terms below, and two more give the uniform remainder,
\[
 \int_0^1 d_b'(y)\cos(Ky)\,dy
 =\frac{(-1)^k d_b''(1)-d_b''(0)}{K^2}+O(K^{-4}).
\]
Using \(b^2=A^2\sin^2 b\) in \eqref{spec:correction} now gives the
second line of \eqref{spec:gap}.

We now prove the sign for every \(k\). It suffices to take \(p=m,\ q=m+1\),
since consecutive mass differences and their limits both telescope.
Write \(a=2/(m+1)\). We show that \(d_b''<0\) and \(d_b'''>0\) on \(0<y<1\).
For \(m\ge2\), put \(z=r_b(y)^a\). The difference
kernel is \(d_b(y)=\log G(z)\), where
\[
 G(z)=\frac{1-z^{1+1/m}}{1-z}
     =(1+1/m)\int_0^1 [t+(1-t)z]^{1/m}\,dt .
\]
Differentiating the integral on \(0<z<1\) gives \(G'>0,\ G''<0,\ G'''>0\).
The derivatives of \(\log G\) have the same three signs. In particular its
third derivative is \(G'''/G-3G'G''/G^2+2(G'/G)^3>0\).
The inner function \(r_b^a\) is increasing and concave. Its third
derivative is positive in the open interval, since \(a\le2/3\) and
\[
 \frac{d^3}{dx^3}\sin^a x
 =a\cos x\,\sin^{a-3}x\,\bigl(2-3a+a^2\cos^2 x\bigr),
 \qquad 0<x<\pi/2 .
\]
The chain rule now gives \(d_b''<0,\ d_b'''>0\).
For \(m=1\), \(d_b=\log(1+r_b)\) is strictly concave, and the sine
addition formula gives
\[
 d_b'(y)=\frac b2\left(\cot\frac{b(1+y)}2+\tan\frac{b(1-y)}2\right).
\]
Both terms have positive second derivative on \(0<y<1\), so again \(d_b'''>0\).

It follows that \(-d_b''\) is positive and strictly decreasing. Integrating
by parts gives
\[
 K\int_0^1 d_b'(y)\cos(Ky)\,dy
     =\int_0^1 [-d_b''(y)]\sin(Ky)\,dy>0 .
\]
Indeed, the bounds at zero above give \(d_b'(y)\sin(Ky)=O(y^a)\)
and \(d_b''(y)\sin(Ky)=O(y^{a-1})\). For the sign, pair each positive
half-wave of the sine with the negative half-wave that follows it. The
first has greater weight, and any unpaired last half-wave is positive
\cite[Theorem~2.1]{ChoPark}. Equation~\eqref{spec:correction} multiplies
this positive cosine integral by \(2(-1)^{k+1}\) to give the difference
from the limit.

For \(c\le1\) the residues of \(T_q-cT_p\) are positive by
\eqref{spec:positive} and the strict increase. For \(c>1\) they are
eventually negative, since \(\rho_k(p)\to\infty\) by
Proposition~\ref{spec:growth} below, while the differences have a finite
limit. The series converges locally uniformly away from its poles by absolute
convergence at zero. Any Stieltjes representation agreeing on \(s>0\)
would agree in the upper half-plane by analytic continuation. There a
Stieltjes function has nonpositive imaginary part, whereas a negative
residue gives positive imaginary part just above its pole.
\end{proof}

The mass ordering, uniform limit and subtraction criterion also hold for
all real \(0<p<q\) if \(H_x=\int_0^1(1-y^x)/(1-y)\,dy\).
This agrees with \(H_n\) at integers and gives \(H_{qn}=O_q(\log(n+1))\)
by comparison with \(H_{\lceil qn\rceil}\). The same kernel applies.
For the uniform limit alone, \(d_b\) is jointly continuous in \(b,y\)
on the closed rectangle. Its increase on the two end intervals of length
\(\delta\) thus tends uniformly to zero as \(\delta\downarrow0\).
This bounds their contribution in \eqref{spec:correction} since \(d_b'>0\).
On the remaining interval one integration by parts gives \(O_\delta(1/k)\)
uniformly in \(b\). The first correction in the first line of
\eqref{spec:gap} also applies to real \(q>2\) by its proof.
The sign at every index need not extend to arbitrary real dilations.
For \(p=1/2,\ q=1\), one has \(d_b=\log(1+r_b^2)\). At \(b=0\) its
derivative \(2y/(1+y^2)\) is strictly increasing. Pairing \(y\) with
\(1-y\) makes the first cosine integral negative. By continuity it stays
negative for all sufficiently small positive angles.

\begin{proposition}[Growth at a fixed angle]\label{spec:growth}
For fixed \(q,\phi\), with Euler's constant \(\gamma\),
\begin{equation}\label{spec:asym}
 \rho_k(q)=
 \begin{cases}
 2\log k+2\gamma+2\log\!\bigl(\pi q\tan(\phi/2)/\phi\bigr)+o(1),&\phi<\pi,\\
 4\log k+4\gamma+2\log(4q)+o(1),&\phi=\pi .
 \end{cases}
\end{equation}
\end{proposition}
\begin{proof}
The leading term of \(1-r_b(1-v)^{2/q}\) is \(\kappa v^a\), with
\((a,\kappa)=(1,\phi\cot b/q)\) for \(b<\pi/2\) and
\((2,\pi^2/(4q))\) at \(b=\pi/2\). Dividing by \(\kappa v^a\) leaves
an analytic function equal to one at zero. Hence
\(g_q'(1-v)-a/v\) is bounded near zero. It is integrable also at the
other endpoint, where \(g_q'(y)=O(y^{2/q-1})\). Its integral is
\(-\log\kappa\), by the endpoint expansion and \(g_q(0)=0\).
Apply the Riemann--Lebesgue lemma to this difference in
\eqref{spec:positive}, then use
\(\int_0^1(1-\cos(k\pi v))\,dv/v=\log(k\pi)+\gamma-\operatorname{Ci}(k\pi)\)
and \(\operatorname{Ci}(k\pi)\to0\) \cite[\S6.2]{DLMF}. This gives
both constants. General Fourier expansions for logarithmic endpoints
are treated by Sidi \cite{Sidi}.
\end{proof}

\begin{corollary}[Zeros]\label{spec:zeros}
Each \(X_q\), and each \(X_q-cX_p\) with \(p<q,\ c\le1\), has only simple
real zeros, one positive zero in each \((\mu_k,\mu_{k+1})\), \(k\ge1\).
\end{corollary}
\begin{proof}
For any such \(X\), continuation of \eqref{spec:pf} gives
\((\phi u/\sin(\phi u))X(u)=\sum_k\rho_k/(\mu_k^2-u^2)\), with all
\(\rho_k>0\). As a function of \(u^2\) the sum has imaginary part of the
same sign as \(\operatorname{Im}(u^2)\), is positive below the first pole,
and increases from \(-\infty\) to \(+\infty\) between consecutive poles
with positive derivative. The sampling formula excludes zeros at the
poles. This is the interlacing argument of Hurwitz and P\'olya
\cite[\S3]{CCP}.
\end{proof}
For \(N_\phi\) the same expansion has masses of sign \((-1)^{k+1}\).
Indeed, for any positive strictly decreasing integrable weight on
\((0,1)\), integration against \(\sin(k\pi y)\) gives a positive value.
Pair consecutive half-waves. In every pair the positive half has larger
integral in absolute value, and an unpaired last half is positive.
The rescaled kernel \(K_N\) has these properties.

\subsection{Which real dilations give the sign law?}\label{spec:realrange}
Here \(p,q\) are real and \(0<p<q\), with \(H_x\) defined above.
The sign law means that the mass increase is strictly above its limit
at every odd index and strictly below it at every even index, for all
\(0<\phi\le\pi\). Put \(a=2/q,\ c=2/p\), and keep \(b=\phi/2,\ K=k\pi\).
The one function needed to describe the boundary is
\begin{equation}\label{real:C}
 C_k(a,c;b)=(-1)^k\log(c/a)+K\int_0^1
 \log\frac{1-r_b(y)^c}{1-r_b(y)^a}\,\sin(Ky)\,dy .
\end{equation}
Integration by parts gives \(C_k=\int_0^1d_b'(y)\cos(Ky)\,dy\).
Thus \eqref{spec:correction} gives the desired sign exactly when \(C_k>0\).
In \eqref{real:C} we may set \(r_0(y)=y\).

\begin{theorem}[Sharp range for real dilations]\label{real:theorem}
If \(p\ge1\), the sign law holds exactly when \(q\ge2\).
For \(0<p<1\), the two numbers \(u,v\) constructed below depend only
on \(p\). The exact range is
\begin{equation}\label{real:range}
                 q\ge 2/u \qquad\text{and}\qquad q>2/v .
\end{equation}
In particular, the threshold \(Q(p)=2/\min(u,v)\) is continuous and greater
than two. The law holds at \(q=Q(p)\) exactly when \(u<v\).
As \(p\) tends up to one, \(Q(p)\) tends to two. As \(p\) tends down
to zero, \(Q(p)\) tends to infinity. Theorem~\ref{real:middle} shows
that \(u<v\) exactly on one interval \(0<p<p_*\) and gives its endpoint.
Proposition~\ref{real:selector} recovers all determining indices and angles.
\end{theorem}

\emph{Construction of the boundary.}
Fix \(c=2/p>2\). For \(n=2,3,\ldots\), set \(s=1-1/n\) and stop as soon
as \(C_j(s,c;\pi/2)<0\) for some \(1\le j\le n\).
Take \(N\ge n\) from the cutoff below, evaluated at \(a=s\).
For \(1\le k\le N\), form the following two lists of roots \(0<x<s\).
\begin{enumerate}
\item For \(u\), take the roots of \(C_k(x,c;0)=0\).
\item For \(v\), take the roots of \(C_k(x,c;\pi/2)=0\) and the
      \(x\)-values solving \(C_k(x,c;b)=\partial_b C_k(x,c;b)=0\)
      at some \(0<b<\pi/2\).
\end{enumerate}
Use the smallest entry in each list, with value \(1\) for an empty list.
Both lists are finite. The distinction between them matters because zero
is a limiting angle, while \(\pi\) belongs to the theorem.
For the initial search, \(C_j(s,c;\pi/2)\) is the integral of \(D'(r)\)
against the polynomial \(\cos(2j\arcsin r)\) on \([0,1]\), by substitution.

\begin{proof}
First suppose \(a\le1\) and \(c\le2\). Split \([a,c]\) at \(1\) if
necessary, and subdivide the part below \(1\) by doubling the lower
endpoint until the next step reaches its upper endpoint. Each piece has
the form \([s,s\lambda]\), with \(s\le1\) and \(1<\lambda\le2\).
The corresponding difference kernels add. For one piece,
\(D(r)=\log G(r^s)\), where
\[
 G(x)=\frac{1-x^\lambda}{1-x}
     =\lambda\int_0^1 w^{\lambda-1}\,dt,\qquad w=t+(1-t)x .
\]
We show that \((1+x)G'(x)/G(x)\) is positive, nonincreasing and convex.
Integration by parts in \(t\), using \((t(1-t))'=1-2t\), gives
\[
 (1+x)G'(x)-(\lambda-1)G(x)
 =\lambda(\lambda-1)(2-\lambda)(1-x)
      \int_0^1t(1-t)w^{\lambda-3}\,dt .
\]
Divide by \(G(x)\). The integral is nonincreasing and convex, as is
\(1/G\) because \(G\) is positive, increasing and concave.
Products of nonnegative, nonincreasing convex functions have these
properties by the product rule. The claim follows.

Substitute \(x=r^s\), an increasing concave function. This preserves
the three properties. To obtain \((1+r)D'(r)\), multiply by
\[
 s\,\frac{(1+r)r^{s-1}}{1+r^s}
 =s\left(1+\frac{r^{s-1}-1}{1+r^s}\right),
\]
which has the same properties. The first factor is at least \(\lambda-1\)
and the multiplier is at least \(s\). Subtracting \(s(\lambda-1)\)
from \((1+r)D'(r)\) therefore leaves a nonnegative, nonincreasing
convex function. Substitute \(r=r_b(y)\) and multiply by
\(r_b'/(1+r_b)\), which is positive, strictly decreasing and strictly convex by the
cotangent and tangent formula above. At \(b=0\) it is \(1/(1+y)\).
This leaves \(d_b'-s(\lambda-1)r_b'/(1+r_b)\) nonnegative,
nonincreasing and convex. Near zero the two logarithms give
\(d_b'=O(y^{s-1})\), \(d_b''=O(y^{s-2})\).
The same integration by parts and pairing of half-waves now give
\[
 C_k(s,s\lambda;b)\ \ge\ s(\lambda-1)C_k(1,2;b)>0 .
\]
Sum over the pieces. Also \(C_k(a,c;b)-C_k(a+h,c;b)=C_k(a,a+h;b)\).
Divide by \(h>0\) and let \(h\downarrow0\) to obtain
\(\partial_a C_k(a,c;b)\le-C_k(1,2;b)<0\) whenever \(a<1,\ c>a\).
Differentiation under the integral is justified locally by an integrable
multiple of \(y^{a/2-1}(1+|\log y|)\).

For the exclusions, a continuous function on \([0,1]\) with nonnegative
cosine coefficients has its maximum at zero. Indeed, every Fej\'er sum
of its even periodic extension has this property, and the sums converge
uniformly. If \(a>1\), then \(d_b'(0)=0\) while \(d_b'>0\) inside.
If \(a=1,\ c>2\), write \(A=b/\sin b\). Expanding at zero gives
\(d_b'(y)=A+A^2y+o(y)\), so again zero is not a maximum.
In both cases some coefficient is negative at each positive angle.

It remains to treat \(c>2,\ 0<a<1\).
At \(a=1,\ b=\pi/2\), the exclusion just proved gives a strictly negative
coefficient. Continuity in \(a\) shows that the search for \(s\) stops.

\emph{Cutoff for the indices.} For \(0<a<1,\ c>2\), choose \(h=(a+\min(2a,1))/2\). As in the proof
of \eqref{spec:gap}, put \(R=d_b'-aA^a y^{a-1}\), where
\(A=b/\sin b\) and \(A=1\) at \(b=0\). Let \(M\) be the larger of
the maxima of \(y^{1-h}|R|\) and \(y^{2-h}|R'|\) on
\(0\le b\le\pi/2,\ 0\le y\le1\). The power series used above shows
that both extend continuously by zero at \(y=0\), since \(h<2a,c\).
Any upper bound for this elementary maximum also works.
Define
\[
 E=a(\pi/2)^a+M\bigl(1/h+1+1/(1-h)\bigr),\qquad
 L=\left(\frac{E}{\Gamma(1+a)\cos(\pi a/2)}\right)^{1/(h-a)} .
\]
Then \(C_k(a,c;b)>0\) for every \(K=k\pi>L\).
Indeed, integrating \(R\cos(Ky)\) up to \(1/K\) costs at most
\(MK^{-h}/h\). One integration by parts on \([1/K,1]\) gives at most
\(MK^{-h}(1+1/(1-h))\), since \(\sin K=0\).
For the leading power, one integration by parts bounds the cosine
integral of \(y^{a-1}\) on \([1,\infty)\) by \(1/K\) in absolute value.
The cosine integral on \([0,\infty)\) now gives
\(C_k(a,c;b)\ge\Gamma(1+a)\cos(\pi a/2)K^{-a}-EK^{-h}\).
Use \(N=\max(n,\lceil L/\pi\rceil)\) at \(a=s\). For \(k>N\) the margin
proved above gives \(C_k(a,c;b)>0\) throughout \(0<a\le s\).

For the remaining indices take the minimum of \(C_k(a,c;b)\) over
\(1\le k\le N\) and \(0\le b\le\pi/2\). This is a continuous, strictly
decreasing function of \(a\) on \((0,s]\). Strict decrease follows from
the common positive minimum of \(C_k(1,2;b)\) over this compact set.
The minimum is negative at \(s\) and tends to \(+\infty\) as \(a\downarrow0\).
To see the latter, choose \(\delta>0\) so that all \(N\) cosines are
at least \(1/2\) on \([0,\delta]\). The integral of \(d_b'\) over this
interval is \(D(r_b(\delta))\ge D(\delta)\), which tends to infinity.
Its integral over the rest is at most \(\log(c/a)-D(\delta)\), which stays
bounded since \(1-\delta^a\sim-a\log\delta\).
Hence the minimum has a unique zero \(a_*\) strictly between zero and \(s\).

We justify the use of finite root lists. The function in \eqref{real:C}
is real analytic in \(a>0\) and in \(b^2\) near the physical interval.
Indeed \(r_b(y)/y\) is analytic and positive there, and at \(r=1\) the
factor \(1-r\) cancels in the ratio. Near \(y=0\) the logarithm series
converges uniformly on a small complex neighborhood of each parameter
pair with \(\Re a>0\). These facts allow integration of the local analytic
expansions. No \(C_k\) is identically zero as a function of the angle.
To see this, continue to \(b=it\), \(t>0\). Now
\(r_b(y)=\sinh(ty)/\sinh t\) tends to zero for \(y<1\).
The logarithm stays between zero and \(\log(c/a)\), so \eqref{real:C}
tends to \((-1)^k\log(c/a)\ne0\).

Since \(\partial_a C_k<0\), near any zero the equation \(C_k(x,c;b)=0\)
defines \(x\) as an analytic function of \(b\). Its stationary points are
isolated, as a constant branch would give an identically zero angular
function. All candidate roots lie away from \(x=0\) by the preceding
uniform divergence. Compactness, including analytic continuation across
\(x=s\) and both angular endpoints, now gives finiteness.
At \(a_*\) a zero must occur at an endpoint or an interior minimum in
\(b\), so \(a_*=\min(u,v)\). For \(a<a_*\) every coefficient is positive.
For \(a>a_*\) up to \(1\) a coefficient is negative by strict monotonicity,
and a negative value at \(b=0\) persists at positive angles. At \(a_*\)
the strict law holds precisely when all contacts are at \(b=0\).
This proves \eqref{real:range}.

For \(c\) near any fixed \(c_0>2\), the same witness at \(s\) stays negative,
and the power series bounds give a common finite \(N\). The minimum of these
\(N\) coefficients over \(b\) is continuous in \(a,c\). Its opposite signs
on either side of \(a_*(c_0)\) persist nearby, proving continuity of the root.
Finally fix \(0<a<1\). At \(c=2\) every coefficient is positive, including
at \(b=0\). The cutoff is bounded for \(2\le c\le3\) at this fixed \(a\),
by the same uniform power series bounds. The finite remaining coefficients
stay positive for \(c\) close to two by compactness. Thus \(Q(p)\to2\) as
\(p\uparrow1\). As \(c\to\infty\), the first coefficient at \(b=0\) is at
most \(-\log(c/a)+\pi\int_0^1[-\log(1-y^a)]\sin(\pi y)\,dy\), which
tends to \(-\infty\). It is then negative at some positive angle, giving
\(Q(p)\to\infty\) as \(p\downarrow0\).
\end{proof}

\subsection{The pole and angle at the boundary}
Fix \(0<p<1\). Call \((k,b)\) a \emph{contact} if
\(C_k(2/Q(p),2/p;b)=0\), with \(0\le b\le\pi/2\).
The angle in the original series is \(\phi=2b\).
Equality means that the strict odd-even law holds at \(q=Q(p)\).
A contact at zero is harmless because the sign law concerns positive angles.
To test equality, first find the earliest zero at limiting angle zero.
If it exists, test every candidate index at that same exponent for a
zero or a negative value at a positive angle. Equality holds precisely
when the first step finds a root and the second finds none.
Proposition~\ref{real:equaltest} gives the proof.

\begin{proposition}\label{real:smallp}
For all sufficiently small \(p>0\), the unique contact is \(k=1,\ b=0\).
Thus equality is allowed, and \(Q(p)=2/a\) where \(C_1(a,2/p;0)=0\).
More precisely,
\[
 pQ(p)\longrightarrow 4\exp\!\left(
 \pi\int_0^1\log(-\log y)\sin(\pi y)\,dy\right)
 \qquad(p\downarrow0).
\]
\end{proposition}
\begin{proof}
Write \(T=-\log r_b(y)\) and \(H_b=-\log T\). As \(c\to\infty,\ a\to0\),
the integral of the absolute difference between \(d_b\) and \(H_b-\log a\)
tends to zero uniformly in \(b\). Indeed their difference is
\(\log(1-r_b^c)+\log[aT/(1-e^{-aT})]\).
The absolute integral of the first term tends to zero by comparison with
\(-\log(1-\sin^c(\pi y/2))\), using \(y\le r_b(y)\le\sin(\pi y/2)\).
The second lies between zero and \(\log(1-a\log y)\).
Both statements follow by dominated convergence.
For fixed odd \(k\), \eqref{real:C} therefore gives
\[
 C_k=-\log(ac)+K\int_0^1 H_b(y)\sin(Ky)\,dy+o(1).
\]
For fixed even \(k\) the constant term is \(\log(c/a)\) instead.
The integrals are uniformly bounded since \(H_0\le H_b\le H_{\pi/2}\)
and both outer functions are integrable.
Set \(J=\pi\int_0^1H_0(y)\sin(\pi y)\,dy\). Testing \(a=\lambda/c\)
with fixed \(\lambda\) on either side of \(e^J\), and using strict
monotonicity, gives a unique first-mode root \(a=a(c)\) at \(b=0\)
with \(ac\to e^J\). We use this root from now on.
At \(b>0\), \(r_b(y)>y\) inside the interval, so the first mode is
strictly positive by \eqref{real:C}.

For an odd \(k\), replacing \(H_b\) by \(H_0\) decreases the sine integral
or leaves it unchanged. Indeed \(T/(-\log y)\) is the average of
\(bu\cot(bu)\) on \([y,1]\) with weight \(du/u\). This average decreases
with \(y\), so \(H_b-H_0\) is nonnegative and increasing. Pair each negative
half-wave with the positive one that follows it. Next,
\(F(y)=H_0(y)+H_0(1-y)\) is strictly convex. With \(t=-\log y\),
\(H_0''(y)=(1-t)/(y^2t^2)\ge-1/(4y^2)\).
For \(y\le1/2\), \(-\log(1-y)\le\min(\log2,y/(1-y))\) gives
\(H_0''(1-y)\ge(1-\log2)/y^2>1/(4y^2)\).
For odd \(k\ge3\) the function
\(\sin(\pi y)/\pi-\sin(Ky)/K\) is positive on \(0<y<1\), by the
finite geometric sum for \(\sin(kx)/\sin x\). It vanishes to third order at
both ends. Integrating twice against \(F''\) and using reflection gives
\(K\int H_0\sin(Ky)\,dy>J\), with integrals over \([0,1]\).
Thus every fixed \(k\ge2\) has a positive margin uniformly in the angle.

We must also exclude indices tending to infinity with \(c\).
Write \(f=d_b'\). Uniformly in the angle and the exponents,
\[
 |f|\le C(y^{-1}+(1-y)^{-1}),\qquad
 |f'|\le C(y^{-2}+(1-y)^{-2}).
\]
Indeed \(f=(-T')[a/(e^{aT}-1)-c/(e^{cT}-1)]\). Each term in brackets is
at most \(1/T\), and its \(T\)-derivative has magnitude at most \(1/T^2\).
On \(y\le1/2\), \(T\) is bounded below, \(|T'|\le1/y\), and \(T''=O(y^{-2})\).
On \(y\ge1/2\), \(1\le T''=b^2/\sin^2(by)\le\pi^2\).
Integrate from \(1\), where \(T=0,\ T'\le0\), to obtain
\(|T'|/T=O((1-y)^{-1})\) and \(T''/T=O((1-y)^{-2})\).

Put \(\delta=1/k\). Replacing the cosine by its endpoint values on
the two end intervals of length \(\delta\) costs \(O(1)\) by its quadratic
error. One integration by parts on the middle interval also gives \(O(1)\).
Consequently \(C_k\ge d_b(\delta)+d_b(1-\delta)-\log(c/a)-O(1)\).
For large \(k\), \(d_b(\delta)\ge-\log(2a\log k)\), since
\(\delta\le r_b(\delta)\le\pi\delta/2\). Also
\(\log(c/a)-d_b(1-\delta)\le\log(1+2c/k)\), by \(1-e^{-x}\ge x/(1+x)\)
and \(T(1-\delta)\le2/k\). If \(k\to\infty\) with \(k\le c^4\), these
bounds give \(C_k\to+\infty\), since \(ac\to e^J\).

For \(k>c^4\) put \(g=-\log(1-r_b^a)\). Fix \(\eta>0\) so small that
\(r_b(\eta)<e^{-4}\) for all \(b\). On \((0,\eta)\), \(g'\) is positive
and decreasing. Indeed the derivative of \(r^{a-1}/(1-r^a)\) has the sign of
\(a-1+r^a<0\) for small \(a\), and \(r_b'\) decreases.
Set \(z=1/(6k)\). Up to \(z\), the integral of \(g'\cos(Ky)\) is at
least \(r_b(z)^a/2\). From \(z\) to \(\eta\) its magnitude is at most
\(2g'(z)/K=O((a+1/\log k)r_b(z)^a)\), by monotonicity,
\(r_b'/r_b\le1/y\) and \(1/(1-e^{-x})\le1+1/x\).
The remaining terms contribute \(O(c^2/k)\) by integration by parts.
Indeed on \([0,\eta]\) the first two derivatives of
\(\log(1-r_b^c)\) are \(O(c^2)\). On \([\eta,1]\), the bracket in
\(f\) equals \(c h(cT)-a h(aT)\), where
\(h(x)=1/x-1/(e^x-1)\) has bounded value and derivative on \([0,\infty)\).
Thus \(f,f'=O(c^2)\) there. Since \(r_b(z)\ge z\) and
\(c^2 k^{a-1}\to0\) for \(k>c^4\), the positive initial term dominates.
Every possible sequence of competing contacts has now been excluded,
proving the assertion and the limit.
\end{proof}

\begin{proposition}\label{real:firstexcluded}
There is a unique \(c_0\in(3,4)\) satisfying \(C_1(1,c_0;0)=0\).
For \(2/c_0\le p<1\) the first pole cannot determine \(Q(p)\).
\end{proposition}
\begin{proof}
Put \(J(c)=C_1(1,c;0)\). Differentiating \eqref{real:C} with respect
to \(\log c\) gives
\[
 -1+\pi\int_0^1\frac{-y^c\log(y^c)}{1-y^c}\sin(\pi y)\,dy .
\]
This decreases strictly in \(c\), since \(-z\log z/(1-z)\) increases
from zero to one on \(0<z<1\). Thus \(J\) is strictly concave in
\(\log c\), and \(J(1)=0\).

To locate its second zero, set \(f(y)=(\log[(1-y^c)/(1-y)])'\).
Then \(J(c)\) is the integral of
\([f(y)-f(1-y)]\cos(\pi y)\) over \(0<y<1/2\).
Write \(z=y(1-y)\). For \(c=3\), direct subtraction gives
\[
 f(y)-f(1-y)=\frac{2z(1-2y)}
 {(1+y+y^2)(3-3y+y^2)}>0 .
\]
For \(c=4\), use \((1-y^4)/(1-y)=(1+y)(1+y^2)\). The difference is
\((1-2y)(3z^2-2)/[(2+z)(2-2z+z^2)]<0\).
Hence \(3<c_0<4\), and \(J(c)\ge0\) on \(2\le c\le c_0\).
At the threshold \(a=2/Q(p)<1\), strict decrease in \(a\) gives
\(C_1(a,c;0)>J(c)\). Every positive angle increases \(C_1\) further
by \eqref{real:C}, since \(r_b(y)>y\) and the sine is positive.
\end{proof}

\begin{proposition}[Contacts near \(p=1\)]\label{real:nearone}
For all \(p<1\) sufficiently close to one, every index determining
\(Q(p)\) is even and equality is excluded. As \(p\uparrow1\), every such
angle tends to \(\pi\) and every such index \(k\) tends to infinity.
More precisely, with \(\varepsilon=2/p-2\) and \(\delta_*=1-2/Q(p)\),
\[
 \varepsilon\log k\longrightarrow\log(4/3),\qquad
 \varepsilon\log(1/\delta_*)\longrightarrow\log(4/3).
\]
\end{proposition}
\begin{proof}
Write \(a=1-\delta,\ c=2+\varepsilon,\ K=k\pi,\ A=b/\sin b\), with
\(A=1\) at zero. Near \((a,c)=(1,2)\) the first three powers in the
logarithm expansion give
\[
 C_k=\sum_{\lambda=a,\,2a,\,c} w_\lambda\Gamma(1+\lambda)
       \cos(\pi\lambda/2)(A/K)^\lambda
       +(-1)^k d_b''(1)K^{-2}+O(K^{-5/2}),
\]
where the three weights in order are \(1,1/2,-1\). The error is uniform
for \(9/10\le a<1,\ 2<c\le21/10\) and all \(b\).
Indeed, subtract \(w_\lambda\lambda A^\lambda y^{\lambda-1}
e^{-y}(1+y)\) for each power from \(d_b'\). The remaining exponents in
\(d_b\) are at least \(\min(3a,a+2,2c)>5/2\), and \(e^{-y}(1+y)=1+O(y^2)\).
As in the proof of \eqref{spec:gap}, the remainder and its first three
derivatives are \(O(y^{3/2-j})\), \(0\le j\le3\). Split at \(1/K\) and
integrate three times by parts on the upper interval. Only the endpoint
slope times \((-1)^k/K^2\) remains outside the error. For each subtracted
term the Laplace integral on \([0,\infty)\) gives the displayed power
with error \(O(K^{-\lambda-2})\) because the two terms from \(1+y\) cancel at
order \(K^{-\lambda-1}\). Subtracting its tail on \([1,\infty)\) restores
its endpoint slope by the same integrations by parts.

Consider a sequence approaching \((1,2)\) with \(C_k\le0\), allowing
\(b\) to vary. Each fixed index is strictly positive at \((1,2)\) on the
closed angle interval, so \(k\to\infty\). Also \(\delta K\) stays bounded.
Otherwise, along a subsequence the positive \(a\)-term dominates the
negative \(2a\)-term by a factor of order
\(\delta K^{1-\delta}=\delta^\delta(\delta K)^{1-\delta}\to\infty\).
The \(c\)-term is positive and the endpoint term is \(O(K^{-2})\).
Whenever \(\delta K\) stays bounded and \(k\to\infty\), one has
\(\delta\log K\to0\). The endpoint formula for \(\log(1+r_b)\) and
continuity give \(d_b''(1)/A^2=-(1+\sin^2 b)/4+o(1)\) uniformly. Hence
\[
 K^2 C_k/A^2=\frac{\pi\delta K}{2A}-1+2e^{-\varepsilon\log K}
              -(-1)^k(1+\sin^2 b)/4+o(1).
\]
Set \(\delta=e^{-t/\varepsilon}\) for fixed \(t>0\).
If \(t<\log(4/3)\), a nonpositive coefficient would give
\(\varepsilon\log K\le t+o(1)\), whereas the last display is at least
\(2e^{-t}-3/2+o(1)>0\). If \(t>\log(4/3)\), choose
\(\log(4/3)<t'<t\), take even \(k\) with \(K\sim e^{t'/\varepsilon}\),
and set \(b=\pi/2\). The display tends to \(2e^{-t'}-3/2<0\).
These two tests bracket \(\delta_*\) by Theorem~\ref{real:theorem}.
At any boundary contact \(C_k=0\) and
\(\varepsilon\log K\le\varepsilon\log(1/\delta_*)+o(1)=\log(4/3)+o(1)\).
The display then forces even parity, \(b\to\pi/2\), and
\(\varepsilon\log K\to\log(4/3)\). Such contacts occur at positive
angles for \(p\) sufficiently close to one.
\end{proof}

\begin{proposition}[Two neighboring even indices]\label{real:neighbors}
Near \(p=1\), at most two indices determine \(Q(p)\). If two do, they
are consecutive even integers. Each has a unique determining angle.
Every sufficiently large even index occurs, and ties accumulate at one.
\end{proposition}

Here is a rule using \(p\) alone. Set \(B=\pi/2,\ L=\log(4/3)\) and
\(\varepsilon=2/p-2\). Choose a smooth \(\chi\) equal to one near zero
and zero near one with \(\chi(y)+\chi(1-y)=1\).
For \(K=e^t,\ A=b/\sin b\), put
\[
 F(a,t,b)=K^2 A^{-2}\int_0^1\chi(y)
       \{d_b'(y)+d_b'(1-y)\}\cos(Ky)\,dy .
\]
Fix a sufficiently small \(\eta>0\). On
\(|\varepsilon t-L|\le\eta,\ B-\eta\le b\le B\), let \(s(t,b)\)
solve \(F=0\) with \(|1-s|\le e^{-t}\).
At its minimum \((t_0,b_0)\), round \(e^{t_0}/(2\pi)\) down and up.
Double both integers. For each resulting \(k\) in the frequency interval
put \(t=\log(k\pi)\). Take \(b=B\) if \(F_b(s(t,B),t,B)\le0\).
Otherwise take the unique interior root of \(F_b(s(t,b),t,b)=0\).
The least resulting \(s(t,b)\) is \(2/Q(p)\). Keep each index attaining
this value and its angle \(2b\).

\begin{proof}
At even \(k\), reflection gives \(F=K^2C_k/A^2\).
With \(\ell=\log A\), the expansion in Proposition~\ref{real:nearone}
now reads
\[
 F=U(t-\ell)+g(b)+R,\qquad
 g=D''(1)\cos^2 b-D'(1)\sin^2 b .
\]
Here \(U(z)\) sums \(w\Gamma(1+\lambda)\cos(\pi\lambda/2)e^{(2-\lambda)z}\)
over the same three weighted powers. Uniformly near \((a,c,b)=(1,2,B)\),
\(R=O(e^{-t/2})\), also after two derivatives in \(t,b\) or one in \(a\).
Indeed, extend the folded integrand by zero beyond one. In addition
to the three subtractions in that proof, subtract the constant and
linear terms of \(d_b'(1-y)\) with the same factor \(e^{-y}(1+y)\).
Their leading cosine terms give \(d_b''(1)/K^2\).
The remainder has \(m\)-th derivative \(O(y^{3/2-m})\) at zero because
the omitted exponents in \(d_b\) exceed \(5/2\) by a fixed margin.
This absorbs the logarithm from an \(a\)-derivative. Angular derivatives preserve
this bound. Differentiating the cosine in \(t\) introduces factors
\((Ky)^j,\ j\le2\). Split at \(1/K\) and integrate \(j+3\) times by
parts above it to retain \(O(K^{-5/2})\). The explicit Laplace
transforms allow the same differentiations.

Set \(u=(1-a)e^t\). For \(|u|\le1\), uniformly on the rectangle,
\[
 F=\pi u/(2A)-1+2e^{-\varepsilon t}-(1+\sin^2 b)/4+o(1),
 \qquad \partial_u F=\pi/(2A)+o(1).
\]
Thus \(s\) exists uniquely and continuously, with \(|u|\le C\eta+o(1)\).
All determining points lie in the rectangle with \(u=o(1)\) by
Proposition~\ref{real:nearone} and its last expansion. Consequently
\(a_0=\min s\le a_d=2/Q(p)<1\).

Hold \(a<1\) fixed on a segment whose endpoints have
\(0\le(1-a)e^t\le C\eta+o(1)\). This bound holds throughout.
There \(|U'|=O(\eta)+o(1)\), and \(U''\ge c_1\varepsilon^2\) for some
\(c_1>0\). The first term of \(U\) is nonnegative and \(O(\eta)+o(1)\).
The second has second derivative \(O((1-a)^2)\). The third is bounded
below by a positive constant and gains a factor \(\varepsilon^2\)
after two derivatives. Also \(g''\to-\cos(2b)/2\). Hence
\[
 d^2 F=U''(dt-\ell'\,db)^2+(g''-U'\ell'')(db)^2
            +O(e^{-t/2})\{(dt)^2+(db)^2\}
\]
is positive definite for small fixed \(\eta\), since \(e^{-t/2}=o(\varepsilon^2)\).
Two minimizers of \(s\) would give \(F(a_0,t,b)<0\) between them,
and thus a smaller root. This proves uniqueness. At level \(a_d\),
\(F\le0\) at this minimum and \(F=0\) at any determining point.
Strict convexity forbids an intervening even frequency, since all
sampled values are nonnegative. If \(e^{t_0}\) is sampled, \(a_d=a_0\)
and uniqueness applies. The same argument at a single sampled frequency
excludes two determining angles.

Both candidates satisfy \(\varepsilon t\to L\), being within \(2\pi\)
of a determining frequency. At fixed \(t\) and any root level,
\(|U'|+|U''|=O(\eta)+o(1)\), so \(F_{bb}>0\) throughout the angle
interval. Any stationary root therefore gives its unique minimum.
The endpoint slope decides whether \(B\) minimizes. The other endpoint
has limiting \(u=-A\cos^2(B-\eta)/(2\pi)<0\), compared with zero at \(B\).
This proves the angle test. A smaller sampled root than \(a_d\) is
excluded by Theorem~\ref{real:theorem}.

Finally the determining indices are locally bounded and each occurs
on a relatively closed set, by the cutoff, continuity of \(Q\), and
compactness of the angle interval. Skipping a large even index would
separate a connected interval into two nonempty relatively closed
sets with lower and higher indices, since none can tie across it.
Both occur because the indices tend to infinity. Without ties on a
terminal interval the index would likewise be locally constant there.
\end{proof}

\begin{corollary}[Order of the changes]\label{real:changes}
There are strictly increasing transition values \(p_k\uparrow1\) for
even \(k\) sufficiently large. Each sufficiently large even \(k\)
determines \(Q(p)\) exactly on \([p_{k-2},p_k]\), and is the sole
determining index in its interior.
At \(p_k\) the two determining indices are \(k,k+2\).
\end{corollary}
\begin{proof}
Use the rectangle of Proposition~\ref{real:neighbors}. Near an active
tie, the minimum of \(s(\log(j\pi),b)\) in \(b\), for \(j=k,k+2\),
has derivative \(V=-F_\varepsilon/F_a\) in \(\varepsilon\), evaluated
at its minimizing angle. Indeed \(F_{bb}>0\) at each fixed root level
gives a unique minimizer, continuous by compactness. Bounding a difference
quotient using the old and new minimizers gives the asserted derivative,
also when the minimizer is an endpoint.

Fix their common level \(a\). Between the two sampled values of \(t\)
let \(b(t)\) minimize \(F(a,t,b)\) in the angle interval.
Here \(u=o(1)\), since this holds at both endpoints.
The expansion gives \(F_{bb}\ge c_0>0\) and \(F_{tb}=o(1)\).
The one-sided minimum conditions at two frequencies and these derivative
bounds give \(|b(t_2)-b(t_1)|\le o(1)|t_2-t_1|\).
At the endpoints \(b(t)\) gives the two determining angles.
Uniformly along this segment,
\[
 F_a\sim F_{ta}\sim-\pi e^t/(2A),\qquad
 F_\varepsilon\sim-2t e^{-\varepsilon t},\qquad
 F_{t\varepsilon},F_{b\varepsilon}=O(1),\quad F_{ba}=O(e^t).
\]
These follow by differentiating the three powers. The proof of the
remainder estimate allows these mixed derivatives, since a logarithm
is absorbed by the same strict margin. Consequently
\(V_t\ge c_2 t e^{-t}\) and \(|V_b|\le C t e^{-t}\).
Thus \(V(t,b(t))\) increases strictly, so the \(k+2\) branch has the
larger \(\varepsilon\)-derivative at the tie. Increasing \(p\) decreases
\(\varepsilon\), and the smaller branch therefore changes from \(k\)
to \(k+2\). Locally no other index enters, by the finite cutoff and
compactness. Single winners persist locally, and ties are isolated.
There are finitely many changes on each compact subinterval near one,
and all are upward steps of two. The indices diverge at one and stay bounded
outside this terminal interval by Proposition~\ref{real:smallp} and the
cutoff on compact intervals. This proves the claim on all of \(0<p<1\).
\end{proof}

\begin{corollary}[When the angle leaves \(\pi\)]\label{real:angleswitch}
For each sufficiently large even \(k\), there is a unique
\(p'_k\in(p_{k-2},p_k)\) such that the determining angle of \(k\) equals \(\pi\)
on \([p_{k-2},p'_k]\) and is strictly between zero and \(\pi\) on
\((p'_k,p_k]\). At \(p_k\) the lower index has an interior angle and the
higher index has angle \(\pi\).
\end{corollary}
\begin{proof}
At an even sample put \(G=F_b+F_t/B\). At \(b=B\) one has
\(G=O(K^{-m})\) for every fixed \(m\), also after one parameter derivative
in \(a\) or \(\varepsilon\). Indeed, in the original cosine integral
\(I(K,b)=\int_0^1 d_b'(y)\cos(Ky)\,dy\), substitution of \(x=by\)
gives \(I_b=-(K/B)I_K\) at \(B\)
because the derivative of \(\sin b\) and the upper endpoint term both vanish.
At even samples the \(K\)-derivatives of \(I\) and its folded integral
differ by
\(\int_0^1\chi(y)d'_B(1-y)\sin(Ky)\,dy\), by reflection.
The weight is smooth, vanishes near one and is odd near zero because
\(r_B(1-y)=\cos(By)\). Repeated integration by parts gives arbitrary
powers of \(1/K\), also for its parameter derivatives. Since
\((\log A)'(B)=1/B\), this proves the claim including the normalization.

At a tie fix its level \(a\) and write \(m(t)=\min_b F(a,t,b)\) on
the strip above. The Hessian bound there is at least a constant times
\(\varepsilon^2\) times the identity. Partial minimization preserves
this lower curvature bound by convexity. Between the two zeros of \(m\),
separated by \(\Delta t\), its endpoint derivatives are therefore at most
\(-c_1\varepsilon^2\Delta t\) and at least \(c_1\varepsilon^2\Delta t\),
respectively, for some \(c_1>0\). These derivatives are \(F_t\) at the
minimizers. Since \(\Delta t\) is of order \(1/K\), the lower index
cannot minimize at \(B\), where its \(F_b\) would be positive.
The higher cannot minimize inside. Indeed \(G_b=F_{bb}+F_{tb}/B>0\), whereas
an interior minimum would give \(G=F_t/B>G(B)\). Its endpoint slope is strictly
negative.

Throughout the winning interval \(\varepsilon t\to L\) as \(k\to\infty\).
At \(b=B,\ a=s(t,B)\) the expansion thus gives \(u=o(1)\).
At fixed \(t\), differentiation along this root gives
\(dF_t/d\varepsilon=F_{t\varepsilon}+F_{ta}V\sim 2t e^{-\varepsilon t}>0\).
Hence \(F_b=-F_t/B+G\) increases strictly with \(p\).
Its signs at the two ties are opposite by the angle test. Its unique
zero gives \(p'_k\) by that same test.
\end{proof}

\begin{remark}\label{real:prefactor}
In fact, with \(\varepsilon=2/p-2\) and Euler's constant \(\gamma\),
every determining index and angle satisfy
\[
 k\sim \tfrac12 e^{13/6-\gamma}(4/3)^{1/\varepsilon},\qquad
 Q(p)-2\sim \tfrac6\pi e^{\gamma-13/6}\varepsilon(3/4)^{1/\varepsilon},
 \qquad \pi-\phi=O(k^{-1/2}).
\]
At a determining point use \(a=2/Q(p)\) and the notation of
Proposition~\ref{real:neighbors}. The neighboring even frequencies
have \(F\ge0\) at the same \(b\), and their logarithms differ from \(t\) by
\(O(e^{-t})\). Since \(F=0\) and \(F_{tt}\) is bounded between them,
Taylor's formula gives \(F_t=O(e^{-t})\), hence \(U'=O(e^{-t/2})\).
At an interior angle \(F_b=0\) now gives \(g'(b)=O(e^{-t/2})\).
Also \(g'(B)=0\) and \(g''\) is bounded below by a positive constant near \(B\). This proves the
angle estimate, also immediate at \(B\).

Write \(U_1,U_2,U_3\) for the three terms in their original order.
Since \(1-a=O(e^{-t})\), we have \(U_2=-1+O(te^{-t})\) and
\(U_1=\varepsilon U_3+O(e^{-t/2})\) by the estimate on \(U'\).
Also \(g(b)=-D'(1)+O(e^{-t})=-(c-a)/2+O(e^{-t})\).
Substitution in \(F=0\) gives
\[
 \Gamma(3+\varepsilon)\cos(\pi\varepsilon/2)(A/K)^\varepsilon
 =U_3=\frac{3+\varepsilon}{2(1+\varepsilon)}+o(\varepsilon).
\]
Taking logarithms, using \(\Gamma'(3)/\Gamma(3)=3/2-\gamma\)
and \(A\to\pi/2\), proves the first equivalent. Finally
\(U_1\sim3\varepsilon/2\) gives \((1-a)K\sim3\varepsilon/2\),
and \(Q(p)-2\sim2(1-a)\).
\end{remark}

\begin{proposition}[A test for equality]\label{real:equaltest}
For the given \(p\), find \(s,N\) and the least zero-angle root \(u\) as
in Theorem~\ref{real:theorem}. If the root list is empty, equality is
excluded. Otherwise put \(F_k(b)=C_k(u,2/p;b)\) for \(1\le k\le N\).
Equality holds exactly for the \(p\) for which each \(F_k(\pi/2)>0\)
and none of these functions has a zero in \(0<b<\pi/2\).
In this case the determining pairs are precisely \((k,0)\) with \(F_k(0)=0\).
\end{proposition}
\begin{proof}
The first boundary lies at or before \(u\), since every coefficient
decreases strictly with \(a\). If the test passes, all coefficients at \(u\)
are positive at positive angles, and at least one vanishes at zero. Thus
\(2/Q(p)=u\). Conversely, if equality holds, every contact is at zero,
so the zero-angle list is nonempty, its first root is \(2/Q(p)\), and
the test passes. Each angular root test is finite by the analyticity and
nonconstancy proved in Theorem~\ref{real:theorem}.
Equivalently, once \(F_k(\pi/2)>0\) has been checked, the angular
test is the convergence of \(\int_0^{\pi/2}(bF_k'(b)/F_k(b))^2\,db\).
At a zero \(b_0\) of order \(m\), the logarithmic derivative equals
\(m/(b-b_0)+O(1)\). The factor \(b\) cancels this pole precisely at zero.
\end{proof}

\begin{corollary}\label{real:partition}
The set of \(p\) allowing equality is a finite union of intervals and
isolated points. On every closed interval inside \(0<p<1\), one can split
the interval at finitely many points so that on each open piece the
determining indices are fixed and their angles depend analytically on \(p\).
\end{corollary}
\begin{proof}
Locally in \(p\), the proof of Theorem~\ref{real:theorem} supplies a common
cutoff and a closed strip \(0<\alpha\le a\le\beta<1\) containing the
boundary in its interior. On this strip \(C_k\) is jointly real analytic
in \(p,a,b\). The earlier logarithm expansion works also for complex \(c\)
near \(2/p\) with positive real part. For each index, label the candidate
set in coordinates \((p,a,b)\) by the angle type, zero, \(\pi/2\), or
interior. Each set is given on a compact box by analytic equations and
inequalities, including \(\partial_b C_k=0\) in the interior.
Each has finitely many pairs \((a,b)\) above any \(p\), as proved earlier.
Apply analytic cell decomposition for restricted analytic functions
\cite[\S4.2]{vdDM} in the coordinate order \(p,a,b\), respecting these
finitely many sets. On each resulting open parameter interval the possible
\(a\)'s form ordered analytic graphs. Over each graph the candidate angles
also form ordered analytic graphs with fixed labels. The lowest graph
gives all the contacts, including ties. A finite cover proves the claim on
any closed parameter interval inside \((0,1)\). Equality is constant on
each open piece. Propositions~\ref{real:smallp} and~\ref{real:nearone}
settle both ends, so the equality set has finitely many components globally.

\end{proof}

\begin{proposition}[Recovering all determining angles]\label{real:selector}
For any given \(0<p<1\), the following integrals give one polynomial for
each candidate index. Its roots are exactly the determining half-angles for
that index. The constant polynomial \(1\) discards the index. Equality
holds exactly when every polynomial is either \(1\) or \(z\).
\end{proposition}
\begin{proof}
Take \(s,N\) from Theorem~\ref{real:theorem}. Write \(c=2/p,\ B=\pi/2\)
and \(v_+=\max(v,0)\). Define
\[
 G(x)=\sum_{k=1}^N\int_0^B[-C_k(x,c;b)]_+\,db,\qquad
 h=\lim_{n\to\infty}\int_0^s\frac{dx}{1+nG(x)} .
\]
Below and at the boundary \(G=0\). Above it some coefficient is negative
on an interval of angles, so \(G>0\). Dominated convergence now gives
\(h=2/Q(p)\). For each \(k\le N\) put \(F(b)=C_k(h,c;b)\). For \(z>B\)
form
\[
 P_k(z)=\lim_{\eta\downarrow0}\exp\!\left\{\frac{2}{\pi}
 \int_0^B \frac{\eta\,|F'(b)|}{\eta^2+F(b)^2}\log(z-b)\,db\right\}.
\]
To see why this is a polynomial, recall that \(F\ge0\) is analytic with
finitely many zeros. Away from them the integral tends to zero.
On one side of a zero \(b_0\), take a small interval where \(F\) is
strictly monotone. The substitution \(t=F(b)/\eta\) changes the
weight to \(dt/(1+t^2)\), with the limits put in increasing order.
Dominated convergence gives \((\pi/2)\log(z-b_0)\) from this side.
Thus each interior zero contributes \((z-b_0)^2\) to the product, and
each endpoint zero contributes \(z-b_0\). This proves the assertions,
including at a multiple zero. For the degree, replace \(\log(z-b)\) by
\(1\) in the exponent and take its limit. Values at one more point than
the degree specify the polynomial by interpolation. Doubling its roots
gives the angles \(\phi\) in the series.
\end{proof}

Proposition~\ref{real:equaltest} is the shorter way to test equality
alone. The next theorem decides equality for the whole parameter range.

\subsection{The interval of equality}
Write \(B=\pi/2\), so that \(b=B\) means \(\phi=\pi\).
\begin{theorem}[The sharp interval of equality]\label{real:middle}
There is exactly one pair \(0<a_*<1,\ c_*>2\) satisfying
\[
 C_1(a_*,c_*;0)=C_2(a_*,c_*;B)=0 .
\]
Put \(p_*=2/c_*\). Then \(20/79<p_*<40/157\).
For \(0<p<1\) the strict odd-even law at every positive angle holds exactly when
\[
 \begin{cases}
 q\ge Q(p),&0<p<p_*,\\
 q>Q(p),&p_*\le p<1 .
 \end{cases}
\]
Throughout \(0<p<p_*\) the first pole at limiting angle zero alone
determines \(Q(p)\). At \(p_*\) it ties with the second pole at \(\phi=\pi\).
These are the only contacts there. For \(p_*<p\le2/c_2\), where \(c_2\)
is specified in Corollary~\ref{real:secondangle}, only the second pole
at \(\pi\) determines the boundary. The two one-sided slopes of
\(Q\) at \(p_*\) are different.
\end{theorem}
\begin{proof}
Let \(a_1(c)\) be the zero-angle first-pole root when it exists.
For every fixed \(b\), the ratio \(C_2(a_1(c),c;b)/\log(c/a_1(c))\)
increases strictly with \(c\). To see this, set \(G(r)=\cos(\pi r)\),
\(F(r)=\cos(2\pi r_b^{-1}(r))\), and \(w(v)=-(v/(e^v-1))'\).
For either function \(T\), the normalized coefficient is the average of
\[
 P_T(s)=\int_0^\infty T(e^{-v/s})w(v)\,dv
\]
with uniform weight in \(\log s\) between \(\log a\) and \(\log c\).
This follows by differentiating the kernel in \eqref{real:C}.
The first root makes the average of \(P_G\) zero. Since \(P_G\)
decreases from one to minus one, the averaging interval expands at both
ends as \(c\) grows from \(c_0\) of Proposition~\ref{real:firstexcluded}.

The function \(P_F\) is strictly convex as a function of \(P_G\).
Indeed \(F'/G'\) decreases, and \(\log w(e^t)\) is strictly concave
(Lemma~S1 of the supplement). Differentiation expresses \(P_F'/P_G'\)
as the average of \(F'/G'\) with weight
\((-G'(r))(-\log r)w(-s\log r)\). Increasing \(s\) gives relatively
more weight to larger \(r\), so this derivative ratio decreases.
This proves convexity since \(P_G\) decreases. As the averaging
interval expands, its endpoints add mass with mean \(P_G=0\).
Their mean \(P_F\) is the value at zero of the chord through the ends
of the convex graph, strictly above the old mean. This proves the increase.

At \(b=B\), \(F(r)=1-8r^2+8r^4\), hence \(C_2(a,c;B)=H(c)-H(a)\), where
\[
 H(x)=\log x+8\{\psi(1+4/x)-\psi(1+2/x)\},\qquad \psi=\Gamma'/\Gamma .
\]
At \(c=157/20\) and \(79/10\) the second coefficient on the first root
has opposite signs (Lemma~S2). Strict increase gives exactly one crossing.

We must also rule out a contact at another index or angle.
The supplement, \emph{Proofs and exact checks for the transition values},
gives the uniform estimates used here. Results S4--S6 show that for
\(2<c\le3\) all determining angles are positive. For \(3\le c\le8\), at the threshold all indices
\(k\ge2\) are positive at zero angle. If \(c<c_*\), a negative second
coefficient at \(a_1(c)\) rules out the first pole as well. Where this
root does not exist, Proposition~\ref{real:firstexcluded} rules it out.
On \([157/20,79/10]\) all \(k\ge3\) are positive at both test roots,
and the second pole can first vanish only at \(B\).
For \(2702/1000\le c\le157/20\), at the root \(H(c)=H(a)\) the second
coefficient strictly decreases with the angle, and each higher
coefficient is positive at every angle (S12). The first coefficient
is positive by the preceding comparison.
Corollary~\ref{real:secondangle} continues the same branch down to \(c_2\).
Above \(79/10\) the strict increase at each angle keeps the second
coefficient positive at \(a_1(c)\). Up to \(c=12\) the same positivity
for every \(k\ge3\) follows from the estimates in the supplement.

For \(c\ge12\), one has \(a_1(c)<3/c\), so it suffices to test the
larger exponent \(3/c\). A direct integral estimate gives positivity
for every \(k\ge4096\). For the remaining \(k\ge3\) write
\(I(K)=\int_0^K(1-\cos t)\,dt/t\). Increasing \(c\) to \(d\) along
\(a=3/c\) can lower \(C_k\) by at most \(\tfrac32(1/c-1/d)I(k\pi)\).
Indeed its derivative in \(\log c\) is \(P_T(c)+P_T(3/c)\) with
\(T(r)=\cos(k\pi r_b^{-1}(r))\). The first term is at least \(-1\),
and \(1-P_T(s)\le sI(k\pi)/2\) since \(w\le1/2\) and \(r_b'/r_b\le1/y\).
The bounds in S8--S10 cover \(12\le c\le36\) and exceed the entire
remaining loss at \(c=36\).

The finite checks use outward-rounded intervals and derivative bounds
between evaluation points. On compact parameter ranges the infinite
tail is controlled by two integrations by parts,
\[
 (k\pi)^2 C_k=\int_0^1d_b'''(y)(1-\cos(k\pi y))\,dy
                 -(1-(-1)^k)d_b''(1).
\]
Near \(c=2\) this identity also compares each odd coefficient with its
two even neighbors. We can now apply Theorem~\ref{real:theorem}.
Below \(p_*\) the sole contact is at the excluded angle zero. At \(p_*\)
the contact at \(\pi\) breaks the strict sign, and for every larger
\(p<1\) there is a contact at a positive angle. At the tie the strict
chord comparison gives a positive derivative of \(C_2(a_1(c),c;B)\)
because the normalizing factor is positive and the coefficient is zero. Since
\(\partial_a C_2<0\), the two root curves have different slopes there.
Taking \(q=2/a\) and \(p=2/c\) preserves this difference.
\end{proof}
\begin{corollary}[A change of angle at the second pole]\label{real:secondangle}
There is exactly one \(c_2\) in \((2699/1000,2702/1000)\) at which
the second pole changes its determining angle. On this closed
\(c\)-interval the second pole alone determines \(Q(2/c)\).
Its angle is interior for \(c<c_2\) and equals \(\pi\) for \(c\ge c_2\).
The value \(c_2\) is characterized there by
\(C_2(a,c;B)=\partial_b C_2(a,c;B)=0\).
Here \(Q\) has the same slope from both sides.
\end{corollary}
\begin{proof}
On the relevant exponent strip the other poles are strictly positive.
The second coefficient is positive for \(b\le3\pi/8\) and strictly
convex in \(b\) for \(3\pi/8\le b\le B\). Along the root \(C_2(a,c;B)=0\),
its endpoint derivative decreases strictly with \(c\) and changes sign.
A nonpositive derivative gives the unique minimum at \(B\). A positive
derivative gives a smaller value inside, hence an earlier interior
contact as \(a\) increases. The supplement proves these bounds and
checks strict positivity at the lower edge of the exponent strip.
For an integral form of the slope test, substitute \(r=r_b(y)\).
At \(B\), the derivative in \(b\) of
\(y=\arcsin(r\sin b)/b\) is \(-\arcsin(r)/B^2\) for \(r<1\).
The derivative stays bounded up to \(r=1\). Thus the second equation
is equivalent to
\(\int_0^1\arcsin(r)\sin(4\arcsin r)D'(r)\,dr=0\).
The same formula with \(2k\) in place of \(4\), together with
\(C_k(a,c;B)=0\), gives two scalar equations for each angle change
of Corollary~\ref{real:angleswitch} on its stated index interval.
To see the slope agreement, solve \(C_2=0\) locally for \(a\) in terms
of \(b,c\). At the switch the \(b\)-derivative of this root is zero,
and its second \(b\)-derivative is positive. The interior minimizing
angle therefore extends smoothly to \(B\) by the implicit function
theorem. Differentiating the root along it gives the same derivative
as keeping \(b=B\) fixed.
\end{proof}
\begin{corollary}[A change from the second to the fourth pole]\label{real:secondfourth}
There is exactly one \(c_4\) in \((2453/1000,2454/1000)\) for which
\[
 C_4(a,c_4;B)=C_2(a,c_4;b)=\partial_b C_2(a,c_4;b)=0,
 \qquad 0<a<1,\quad 7B/10<b<B .
\]
For \(2453/1000\le c<c_4\) the fourth pole alone determines \(Q(2/c)\),
at angle \(\pi\). On the entire interval \(c_4<c<c_2\), the second pole
alone determines it, at a unique interior angle. Exactly these two pairs
tie at \(c_4\).
\end{corollary}
\begin{proof}
Take the root in \(a\) of \(C_4(a,c;B)=0\). Along it, the minimum
of \(C_2\) over \(7B/10\le b\le B\) strictly decreases with \(c\).
The supplement proves that it is positive at \(2453/1000\) and negative
at \(2454/1000\). On the strip containing both thresholds, \(C_2\) is
strictly convex on this angle interval and positive at both ends and
at smaller angles. The fourth coefficient can first vanish only at \(B\),
and the other poles stay positive. Since both coefficients strictly
decrease in \(a\), the sign of this one minimum decides which reaches
zero first. Its unique zero gives the asserted tie. From \(2454/1000\)
to \(2699/1000\), S14--S15 exclude all other poles up to the first
interior contact and prove strict angular convexity wherever that
contact can occur. Corollary~\ref{real:secondangle} covers the remaining
interval up to \(c_2\).
\end{proof}

Further changes of index or angle type have an exact characterization on any closed
interval inside \(0<p<1\). Use the common finite cutoff of
Corollary~\ref{real:partition}. For each candidate index and each angle
type (zero, interior, or \(\pi\)), mark the parameters for which at least
\(m\) distinct angles of that type give the lowest root in
Theorem~\ref{real:theorem}, for \(m=1,2,\ldots\). Each marked set is a
finite union of intervals and points by the same analytic decomposition.
Only finitely many \(m\) occur. Their boundary points in the interior
of the chosen interval are precisely the changes there. The numbers and types are locally
constant off these boundaries. At any boundary at least one of them
changes, including when a tie occurs at a single parameter.

\section{Angle formulas and cancellations}\label{ang:section}
\subsection{The three angle formulas}
For a real angle \(\alpha\), define
\begin{equation}\label{ang:fourierdef}
 C_\alpha(u)=\sum_{k\geq1}\frac{\cos(k\alpha)}{k(k^2-u^2)},
 \qquad S_\alpha(u)=\sum_{k\geq1}\frac{\sin(k\alpha)}{k^2-u^2},
 \qquad Z(u)=C_0(u).
\end{equation}
These are meromorphic with possible simple poles at the nonzero integers. Indeed, the sums converge locally uniformly away from these integers. For \(|u|<1\), expansion of each denominator into a geometric series gives
\[
 Z(u)=\sum_{j\geq0}\zeta(2j+3)u^{2j},\qquad
 C_\alpha(u)=\sum_{j\geq0}\operatorname{Cl}_{2j+3}(\alpha)u^{2j},
 \qquad
 S_\alpha(u)=\sum_{j\geq0}\operatorname{Cl}_{2j+2}(\alpha)u^{2j}.
\]
Here \(\operatorname{Cl}_r(\alpha)\) is the sum over \(k\geq1\) of
\(\cos(k\alpha)/k^r\) for odd \(r\geq3\), and of \(\sin(k\alpha)/k^r\) for even \(r\geq2\).

\begin{theorem}[Angle formulas]\label{ang:main}
Write \(c=\cos(u\phi)\), \(s=\sin(u\phi)\), and \(\ell=\log(2\sin(\phi/2))\).
For every \(0<\phi\leq\pi\),
\begin{align}
 N_\phi(u)&=\frac{2\ell(1-c)}{u^2}-2Z(u)+2c\,C_\phi(u)
                         +\frac{2s}{u}S_\phi(u),\label{ang:N}\\
 A_\phi(u)&=\frac{2\ell(1-c)}{u^2}+c[Z(u)+C_{2\phi}(u)]-2C_\phi(u)
                         +\frac{s}{u}S_{2\phi}(u),\label{ang:A}\\
 B_\phi(u)&=\frac{\ell(1-c)}{u^2}+4c[Z(2u)+C_{\pi-\phi}(2u)]-C_\phi(u)
                         -\frac{2s}{u}S_{\pi-\phi}(2u).\label{ang:B}
\end{align}
Each full right side takes its limiting value at an apparent pole and at \(u=0\).
\end{theorem}
\begin{proof}
On \(0<\theta<b\), the Fourier expansions of the three kernels are
\begin{align*}
 K_N(\theta)&=2\ell+2\sum_{k\geq1}\frac{\cos(2k\theta)}k,\\
 K_A(\theta)&=2\ell+2\sum_{k\geq1}\frac{\cos(k\phi)\cos(2k\theta)}k,\\
 K_B(\theta)&=\tfrac12 K_A(\theta)
        +2\sum_{\substack{k\geq1\\ k\ {\rm odd}}}
                 \frac{\sin(kb)\sin(k\theta)}k .
\end{align*}
For the first, use \(-\log|2\sin(x/2)|=\sum_{k\geq1}\cos(kx)/k\).
For the second, factor \(\sin^2 b-\sin^2\theta=\sin(b-\theta)\sin(b+\theta)\).
For the third, factor \(\sin b-\sin\theta=2\sin((b-\theta)/2)\cos((b+\theta)/2)\) and separate the even and odd Fourier modes. These expansions converge in \(L^2(0,b)\): the basic log-sine series converges in \(L^2\) on a period, and each translation and rescaling above restricts it to a bounded interval. Thus it can be integrated against \(\sin(2u\theta)\) for any fixed complex \(u\).

First take \(2u\notin\mathbb Z\). Product-to-sum integration gives
\begin{align*}
 \int_0^b\sin(2u\theta)\cos(2k\theta)\,d\theta
  &=\frac{u(c\cos(k\phi)-1)+k s\sin(k\phi)}{2(k^2-u^2)},\\
 \int_0^b\sin(2u\theta)\sin(k\theta)\,d\theta
  &=\frac{2u c\sin(kb)-k s\cos(kb)}{k^2-4u^2}.
\end{align*}
The constant mode integrates to \((1-c)/(2u)\). Inserting these values in \eqref{ang:int} gives \eqref{ang:N}. For \(A_\phi\), use \(2\cos^2(k\phi)=1+\cos(2k\phi)\) and \(2\sin(k\phi)\cos(k\phi)=\sin(2k\phi)\) to get \eqref{ang:A}.
For \(B_\phi\) the odd modes give
\[
 B_\phi(u)-\tfrac12 A_\phi(u)
 =4c\sum_{k\ {\rm odd}}\frac{1-\cos(k\phi)}{k(k^2-4u^2)}
  -\frac{2s}{u}\sum_{k\ {\rm odd}}\frac{\sin(k\phi)}{k^2-4u^2},
\]
where the odd sums start at \(k=1\). Use the odd-mode indicator
\((1-(-1)^k)/2\). The even-mode identities
\begin{equation}\label{ang:dup}
 C_\alpha(2u)+C_{\pi-\alpha}(2u)=\tfrac14 C_{2\alpha}(u),
 \qquad
 S_\alpha(2u)-S_{\pi-\alpha}(2u)=\tfrac12 S_{2\alpha}(u)
\end{equation}
follow directly by putting \(k=2j\) in \eqref{ang:fourierdef}. Applied with \(\alpha=0,\phi\), they reduce the last display to \eqref{ang:B} after adding \(A_\phi/2\). The left sides of \eqref{ang:N}--\eqref{ang:B} are entire by Lemma~\ref{ang:arcsine}. This proves removability and extends the formulas to all \(u\).
\end{proof}

\subsection{Two relations that cancel the sine terms}
\begin{theorem}\label{ang:relations}
With \(c\) and \(\ell\) as in Theorem~\ref{ang:main}, for \(0<\phi\leq\pi\),
\begin{equation}\label{ang:same}
\begin{aligned}
 N_\phi(2u)+c[B_\phi(u)-A_\phi(u)]
 ={}&\frac{\ell(1-c)}{u^2}+2\cos(2u\phi)Z(2u)-c^2Z(u)\\
    &-2C_\phi(2u)+c\,C_\phi(u).
\end{aligned}
\end{equation}
For \(0<\phi<\pi/2\), put \(\phi'=\pi-2\phi\). Then
\begin{equation}\label{ang:two}
\begin{aligned}
 &\sin(2u\phi)B_{\phi'}(u)+4\sin(u\phi')A_\phi(2u)\\
 &\quad=4\sin(\pi u)[Z(2u)+C_{2\phi}(2u)]
            -\sin(2u\phi)C_{\phi'}(u)-8\sin(u\phi')C_\phi(2u)\\
 &\qquad+\frac{\log(2\cos\phi)\sin(2u\phi)[1-\cos(u\phi')]}{u^2}
            +\frac{2\ell\sin(u\phi')[1-\cos(2u\phi)]}{u^2}.
\end{aligned}
\end{equation}
The same continuity convention applies.
\end{theorem}
\begin{proof}
In \eqref{ang:same} the three sine sums have the common multiplier \(2sc/u\), with \(s=\sin(u\phi)\). The remaining factor is
\(S_\phi(2u)-S_{\pi-\phi}(2u)-S_{2\phi}(u)/2=0\) by \eqref{ang:dup}.
The logarithmic terms simplify by
\((1-\cos(2u\phi))/2-c(1-c)=1-c\).
The cosine terms simplify by the first identity in \eqref{ang:dup} and
\(\cos(2u\phi)=2c^2-1\).

For \eqref{ang:two}, evaluate \eqref{ang:B} at \(\phi'\) and \eqref{ang:A} at \(2u\). Since \(\pi-\phi'=2\phi\), both sine sums involve \(S_{2\phi}(2u)\). Their coefficients in the claimed combination are opposites. The coefficient of \(Z(2u)+C_{2\phi}(2u)\) is \(4\sin(u(2\phi+\phi'))=4\sin(\pi u)\). The two logarithms are \(\ell\) and \(\log(2\cos\phi)\), which gives the displayed formula.
\end{proof}
The angle \(\phi'\) has series parameter \(2\sin(\phi'/2)=2-t^2\). The unique fixed point in \(0<t<\sqrt2\) is \(t=1\), or \(\phi=\pi/3\).

\begin{corollary}\label{ang:coefficient}
For \(m\geq0\), the sum
\[
\sum_{n\geq1}\frac{t^{2n}}{n^2\binom{2n}{n}}
 \left[\frac{4^m e_m(n)}{n}+
  (H_{2n}-H_n)\sum_{j=0}^m\frac{\phi^{2j}e_{m-j}(n)}{(2j)!}\right]
\]
equals
\begin{align*}
 &\frac{\phi^{2m+2}\ell}{(2m+2)!}
 +(-1)^m(2^{2m+1}-1)
       [\zeta(2m+3)-\operatorname{Cl}_{2m+3}(\phi)]\\
 &\quad +\sum_{j=1}^m\frac{(-1)^{m+j}\phi^{2j}}{(2j)!}
 \left[(2^{2m+1}-2^{2j-1})\zeta(2m+3-2j)
                   +\operatorname{Cl}_{2m+3-2j}(\phi)\right].
\end{align*}
\end{corollary}
\begin{proof}
Take \((-1)^m[u^{2m}]\) in \eqref{ang:same}. On the left the cosine acts by the finite convolution in the statement. On the right expand \(c^2=(1+\cos(2u\phi))/2\) and use the Taylor series following \eqref{ang:fourierdef}. The terms with no angle power give \((-1)^m(2^{2m+1}-1)[\zeta(2m+3)-\operatorname{Cl}_{2m+3}(\phi)]\).
\end{proof}

At \(m=0\) this gives
\begin{equation}\label{ang:weight3}
 \sum_{n\geq1}\frac{t^{2n}(H_{2n}-H_{n-1})}{n^2\binom{2n}{n}}
   =\zeta(3)-\operatorname{Cl}_3(\phi)+\tfrac12\phi^2\log t .
\end{equation}
For example, let \(\tau=(1+\sqrt5)/2\) and \(L_n=\tau^n+(-1/\tau)^n\) be the Lucas numbers. At \(\phi=\pi/5,3\pi/5\), the values of \(t\) are \(1/\tau,\tau\). Averaging \(\cos(k\alpha)\) over the five angles \((2j+1)\pi/5\), \(0\leq j<5\), gives zero unless \(5\mid k\), and then gives \((-1)^{k/5}\). Hence
\(\operatorname{Cl}_3(\pi/5)+\operatorname{Cl}_3(3\pi/5)=9\zeta(3)/25\).
Adding the two cases of \eqref{ang:weight3} recovers Sun's Lucas series, proved by Xu and Zhao \cite{Sun2015,XuZhao},
\[
 \sum_{n\geq1}\frac{L_{2n}(H_{2n}-H_{n-1})}{n^2\binom{2n}{n}}
   =\frac{41\zeta(3)+4\pi^2\log\tau}{25}.
\]

\subsection{Two more cancellations}
Write \(\ell_\alpha=\log(2\sin(\alpha/2))\) for \(0<\alpha\leq\pi\).
As in Theorem~\ref{ang:main}, put \(c=\cos(u\phi)\) and \(s=\sin(u\phi)\).
Doubling the angle in \(N\) removes the sine series from \(A\).
For \(0<\phi\leq\pi/2\) this gives
\begin{equation}\label{rc:double}
\begin{aligned}
 4cA_\phi(u)-N_{2\phi}(u)
  ={}&\frac{8c(1-c)\ell_\phi-2(1-\cos(2u\phi))\ell_{2\phi}}{u^2}\\
     &+(4c^2+2)Z(u)+2C_{2\phi}(u)-8cC_\phi(u).
\end{aligned}
\end{equation}
For \(0<\phi<\pi\), write \(\beta=\pi-\phi\) and \(v=2u\beta\).
Taking the supplementary angle in \(N\) removes the sine series from \(B\).
The resulting identity is
\begin{equation}\label{rc:supp}
\begin{aligned}
 \sin v\,B_\phi(u)+2sN_\beta(2u)
  ={}&\frac{\ell_\phi(1-c)\sin v+\ell_\beta s(1-\cos v)}{u^2}\\
     &+4(c\sin v-s)Z(2u)+4\sin(v+u\phi)C_\beta(2u)
        -\sin v\,C_\phi(u).
\end{aligned}
\end{equation}
Indeed, the coefficient of \(S_{2\phi}(u)\) in \eqref{rc:double} is
\((4cs-2\sin(2u\phi))/u=0\), and that of \(C_{2\phi}(u)\) is
\(4c^2-2\cos(2u\phi)=2\).
In \eqref{rc:supp} the terms in \(S_\beta(2u)\) have coefficients
\(-2s\sin v/u\) and \(2s\sin v/u\). The addition formula for sine
combines the two coefficients of \(C_\beta(2u)\).
Substitution in \eqref{ang:N}--\eqref{ang:B} gives the remaining terms.
Both identities extend to every complex \(u\) by continuity.

\subsection{Weights three and five}
At \(u=0\) the three angle formulas give the following known evaluations,
in the harmonic-number convention of \eqref{intro:functions}
\cite[eqs.~(2.35), (2.36), (2.66)]{DK}.
\begin{equation}\label{rc:weightthree}
\begin{aligned}
 N_\phi(0)&=\phi^2\ell_\phi-2\zeta(3)+2\Cl_3(\phi)+2\phi\Cl_2(\phi),\\
 A_\phi(0)&=\phi^2\ell_\phi+\zeta(3)+\Cl_3(2\phi)-2\Cl_3(\phi)
            +\phi\Cl_2(2\phi),\\
 B_\phi(0)&=\tfrac12\phi^2\ell_\phi+4\zeta(3)+4\Cl_3(\pi-\phi)
            -\Cl_3(\phi)-2\phi\Cl_2(\pi-\phi).
\end{aligned}
\end{equation}
Here \(C_\alpha(0)=\Cl_3(\alpha)\), \(S_\alpha(0)=\Cl_2(\alpha)\),
\(s/u\to\phi\), and \((1-c)/u^2\to\phi^2/2\).
Write \(d_n=n^2\binom{2n}{n}\).
Corollary~\ref{ang:coefficient} at \(m=1\), using
\(e_1(n)=H_{n-1}^{(2)}\), gives
\begin{equation}\label{rc:weightfive}
\begin{aligned}
 &\sum_{n\ge1}\frac{t^{2n}}{d_n}
  \left[\left(H_{n-1}^{(2)}+\frac{\phi^2}{2}\right)(H_{2n}-H_n)
                      +\frac{4H_{n-1}^{(2)}}{n}\right]\\
 &\qquad=\frac{\phi^4\ell_\phi}{24}
   +7\bigl[\Cl_5(\phi)-\zeta(5)\bigr]
   +\frac{\phi^2}{2}\bigl[\Cl_3(\phi)+6\zeta(3)\bigr].
\end{aligned}
\end{equation}
For odd \(j\ge3\), the even-index terms in the cosine series give
\(\Cl_j(\pi/2)=-2^{-j}(1-2^{1-j})\zeta(j)\), and the odd-index terms
vanish. Consequently \eqref{rc:weightfive} at \(\phi=\pi/2\) becomes
\begin{equation}\label{rc:halffive}
\begin{aligned}
 &\sum_{n\ge1}\frac{2^n}{d_n}
    \left[\left(H_{n-1}^{(2)}+\frac{\pi^2}{8}\right)(H_{2n}-H_n)
                      +\frac{4H_{n-1}^{(2)}}{n}\right]\\
 &\qquad=\frac{\pi^4\log2}{768}
                +\frac{189\pi^2\zeta(3)}{256}-\frac{3689\zeta(5)}{512}.
\end{aligned}
\end{equation}
A different harmonic weight occurs at \(\phi=2\pi/3\). The coefficient
of \(-u^2\) in the second identity of \eqref{fix:simplifications} gives
\begin{equation}\label{rc:thirdfive}
 \sum_{n\ge1}\frac{3^n H_{n-1}^{(2)}}{d_n}\left(2H_n+\frac1n\right)
 =\frac{2\pi^2}{81}\bigl[\pi^2\log3+2\zeta(3)\bigr].
\end{equation}
Indeed, with \(h=2\pi/3\), this coefficient on the right is
\(h^4\log(3)/8+h^2\zeta(3)/9\).

\section{Exact masses at rational angles}\label{mass:section}

Return to \(A_\phi=X_1\) and \(B_\phi=X_2\), with masses \(\rho_k^A,\rho_k^B\). A finite formula is available whenever the sampling frequency \(k\pi/\phi\) is an integer. At the other frequencies a change of frequency by an integer still has a finite formula.

\begin{theorem}[Integer frequencies]\label{mass:finite}
If \(m=k\pi/\phi\) is an integer, put
\[
 J_m(x)=\sum_{j=1}^{m-1}\frac{1-\cos(2jx)}{j}
                 +\frac{1-\cos(2mx)}{2m}.
\]
Then
\begin{align}
 \rho_k^A&=2\bigl[2J_m(\phi/2)-J_m(\phi)\bigr],\label{mass:A}\\
 \rho_k^B&=\tfrac12\rho_k^A+
 2\sum_{j=1}^{m}\frac{1-\cos((2j-1)\phi)}{2j-1}.
 \label{mass:B}
\end{align}
In particular both formulas apply for every \(k\) if \(\phi=\pi/d\), \(d\) a positive integer.
\end{theorem}
\begin{proof}
Write \(b=\phi/2\), \(\mu=k\pi/\phi\). In \eqref{spec:positive} change to the angle \(\theta=b(1-v)\). The trigonometric factor becomes \(1-(-1)^k\cos(2\mu\theta)\). In this variable the derivatives of the kernels for \(A_\phi\) and \(B_\phi-A_\phi/2\) are respectively
\[
 \cot(b-\theta)-\cot(b+\theta),\qquad
 \tfrac12\bigl[\csc(b-\theta)+\csc(b+\theta)\bigr].
\]
The first identity follows by factoring \(\sin^2 b-\sin^2\theta\); for the second, differentiate \(\tfrac12\log[(\sin b+\sin\theta)/(\sin b-\sin\theta)]\).
Define \(I_\mu(x)=\int_0^x(1-\cos(2\mu v))\cot v\,dv\), for \(0\le x\le\phi\). Substituting \(v=b-\theta\) and \(v=b+\theta\) on the two halves of the interval gives
\begin{equation}\label{mass:integrals}
 \rho_k^A=2[2I_\mu(b)-I_\mu(\phi)],\qquad
 D_k:=\rho_k^B-\tfrac12\rho_k^A
       =\int_0^\phi\frac{1-\cos(2\mu v)}{\sin v}\,dv .
\end{equation}
At zero the numerator vanishes to second order, so both integrands extend continuously. For \(\phi=\pi\), the numerator also has a double zero at the upper endpoint. At integer \(\mu=m\) one has
\[
 \frac{1-\cos(2mv)}{\sin v}
   =2\sum_{j=1}^m\sin((2j-1)v).
\]
Multiplying by \(\cos v\) gives \(2\sum_{j=1}^{m-1}\sin(2jv)+\sin(2mv)\). Integrating the two finite series gives \eqref{mass:B} and \(I_m=J_m\), respectively.
\end{proof}

\begin{proposition}[Finite recurrences]\label{mass:rec}
Let \(\phi=h\pi/d\) in lowest terms and \(\mu=k\pi/\phi\). With \(f_r(x)=(1-\cos(2rx))/(2r)\), one has
\begin{align}
 D_{k+h}-D_k&=2\sum_{j=1}^{d}
    \frac{1-\cos((2j-1)\phi)}{2\mu+2j-1},\label{mass:stepD}\\
 I_{\mu+d}(x)-I_\mu(x)&=\sum_{j=1}^{d}\bigl[f_{\mu+j-1}(x)+f_{\mu+j}(x)\bigr],
 \qquad x=b,\phi .\label{mass:stepI}
\end{align}
Together with \eqref{mass:integrals}, these determine each sequence of masses from its first \(h\) values. Equation \eqref{mass:integrals} gives the starting values even at noninteger frequencies.
\end{proposition}
\begin{proof}
The step \(k\mapsto k+h\) sends \(\mu\) to \(\mu+d\). Telescope the difference of the cosines in \eqref{mass:integrals} into \(d\) steps. In one step, division by \(\sin v\) gives \(2\sin((2\mu+2j-1)v)\). On integrating to \(\phi\), use \(2\mu\phi=2k\pi\). An extra factor \(\cos v\) in the same calculation gives \eqref{mass:stepI}.
\end{proof}

\paragraph{Starting values.}
One can also compute the first \(h\) masses by a finite formula in the digamma function \(\psi=\Gamma'/\Gamma\). For an angle \(\alpha\) with \(R\alpha\in2\pi\mathbb Z\), take \(R\) a positive integer. Then
\begin{equation}\label{mass:seed}
 u^2 C_\alpha(u)=\frac{1}{2R}\sum_{j=1}^R\cos(j\alpha)
 \left[2\psi\!\left(\frac jR\right)
 -\psi\!\left(\frac{j-u}{R}\right)-\psi\!\left(\frac{j+u}{R}\right)\right].
\end{equation}
Indeed, split \(u^2/[k(k^2-u^2)]\) into \(\tfrac12[(k-u)^{-1}+(k+u)^{-1}-2/k]\) and sum in residue classes using the partial fraction series for \(\psi\) \cite[\S5.7]{DLMF}. This proves \eqref{mass:seed} away from the poles and hence as a meromorphic identity. In \eqref{ang:A} and \eqref{ang:B} the sine terms vanish at \(u=\mu_k\). Even at a possible pole of \(S_{2\phi}(u)\) or \(S_{\pi-\phi}(2u)\), its resonant numerator vanishes: \(\sin(2\mu_k\phi)=0\), and \(\sin(2\mu_k(\pi-\phi))=0\) if \(2\mu_k\) is integral. The remaining terms are cosine sums, including \(Z=C_0\). Use \eqref{mass:seed} with \(R=2d\), take the limit in the full expression, and apply \eqref{spec:pf}. This evaluates the starting values at rational arguments of \(\psi\), with finite parts at its poles.

\subsection{Three special angles}
Write \(\Omega_j=\sum_{\ell\le j,\ \ell\ {\rm odd}}1/\ell\), with \(\ell\ge1\), and let \(\delta_k=1\) for odd \(k\) and zero otherwise. At \(\phi=\pi/3\) one obtains
\begin{align}
 \rho_k^A&=3H_{3k}-H_k-4(\Omega_{3k}-\Omega_k)-\frac{4\delta_k}{3k},
 &\rho_k^B&=\tfrac12\rho_k^A+\Omega_{6k}+\Omega_{2k} .\label{mass:third}
\end{align}
Indeed, in \eqref{mass:A} use \(\cos(2j\pi/3)=(3\mathbf1_{3\mid j}-1)/2\) and \(\cos(j\pi/3)=(-1)^j\cos(2j\pi/3)\), keeping the half-weight at \(j=3k\). In \eqref{mass:B}, the cosine is \(1/2\) on odd indices prime to \(3\) and \(-1\) on odd multiples of \(3\).

At \(\phi=\pi/2\) the cosine in \eqref{mass:B} vanishes at every odd index. Therefore
\begin{equation}\label{mass:half}
 \rho_k^A=4\Omega_k-\frac{2\delta_k}{k},\qquad
 \rho_k^B=\tfrac12\rho_k^A+2H_{4k}-H_{2k}.
\end{equation}
For the first identity, separate even and odd \(j\) in \eqref{mass:A}, using \(\cos(j\pi/2)=0\) for odd \(j\) and \(\cos(j\pi/2)=(-1)^{j/2}\) otherwise. For the second, \(2\Omega_{4k}=2H_{4k}-H_{2k}\).

At every integer frequency \(m=k\pi/\phi\) the binomial evaluations themselves are finite:
\begin{equation}\label{mass:binfinite}
 \sum_{n=1}^{m}t^{2n}w_n(m)H_{qn}
       =\frac{(-1)^{k+1}}{2m^2}\rho_k(q).
\end{equation}
This follows directly from \eqref{intro:w} and \eqref{spec:pf}. For \(q=1,2\), Theorem~\ref{mass:finite} replaces the products in this sum by single trigonometric sums.
For the endpoint \(\phi=\pi\), \eqref{mass:A}--\eqref{mass:B} give
\[
 \rho_k^A=8\Omega_k-\frac{4\delta_k}{k},\qquad
 \rho_k^B=\tfrac12\rho_k^A+4\Omega_{2k}.
\]
In fact \(J_k(\pi)=0\), and \(J_k(\pi/2)=2\Omega_k-\delta_k/k\). At this angle \(K_N(\theta)=K_A(\pi/2-\theta)\); reflecting \(\theta\) in \eqref{ang:int} gives \(N_\pi(k)=(-1)^{k+1}A_\pi(k)\). Hence the finite sum in \eqref{mass:binfinite} with weight \(1/n\) instead of \(H_{qn}\) equals \(\rho_k^A/(2k^2)\).

\subsection{Terminating evaluations}
Write \(d_n=n^2\binom{2n}{n}\). At \(\phi=\pi/3\) one has \(t=1\) and \(\mu_k=3k\).
For \(k\ge1\), \eqref{mass:binfinite} gives both finite sums
\begin{equation}\label{rc:terminating}
\begin{aligned}
 \sum_{n=1}^{3k}\frac{H_n}{d_n}\prod_{j=1}^{n-1}
   \left(1-\frac{9k^2}{j^2}\right)
 &=\frac{(-1)^{k+1}\rho_k^A}{18k^2},\\
 \sum_{n=1}^{3k}\frac{H_{2n}}{d_n}\prod_{j=1}^{n-1}
   \left(1-\frac{9k^2}{j^2}\right)
 &=\frac{(-1)^{k+1}\rho_k^B}{18k^2}.
\end{aligned}
\end{equation}
The masses on the right are the harmonic sums in \eqref{mass:third}.
At \(\phi=\pi\) the endpoint formulas above
and \(N_\pi(k)=(-1)^{k+1}A_\pi(k)\) give
\begin{equation}\label{rc:termend}
\begin{aligned}
 \sum_{n=1}^{k}4^n w_n(k)H_n
 &=\frac{(-1)^{k+1}}{k^2}\left(4\Omega_k-\frac{2\delta_k}{k}\right),\\
 \sum_{n=1}^{k}\frac{4^nw_n(k)}n
 &=\frac{1}{k^2}\left(4\Omega_k-\frac{2\delta_k}{k}\right).
\end{aligned}
\end{equation}

\begin{lemma}\label{rm:poles}
For every positive integer \(m\), as \(u\to m\),
\[
 Z(u)=-\frac{1}{2m^2(u-m)}
        +\frac{1}{m^2}\left(\frac{3}{2m}-H_m\right)+O(u-m).
\]
If \(m\) is odd, then
\[
 Z(m/2)=\frac{4}{m^2}\left(2\log2-2\Omega_m+\frac1m\right).
\]
\end{lemma}
\begin{proof}
The term with index \(m\) in \(Z(u)=\sum_{k\ge1}1/[k(k^2-u^2)]\)
has residue \(-1/(2m^2)\) and constant term \(1/(4m^3)\).
The remaining series is analytic near \(m\) by local uniform convergence.
For \(k\ne m\), write
\[
 \frac1{k(k^2-m^2)}
   =\frac1{m^2}\left(-\frac1k+\frac1{2(k-m)}+\frac1{2(k+m)}\right).
\]
Summing up to \(K>m\) and omitting \(k=m\) gives
\[
 \frac1{m^2}\left(
    -H_K+\frac{H_{K-m}+H_{K+m}}2+\frac{5}{4m}-H_m\right).
\]
The three \(K\)-dependent terms tend to zero. Adding \(1/(4m^3)\)
proves the first formula.

For odd \(m\) the summand at \(u=m/2\) is
\(\frac4{m^2}[-1/k+1/(2k-m)+1/(2k+m)]\).
For \(K>m\), summing the bracket through \(k=K\) gives
\[
 -H_K+\Omega_{2K-m}+\Omega_{2K+m}-2\Omega_m+1/m .
\]
Both odd harmonic sums differ from
\(\Omega_{2K}=H_{2K}-H_K/2\) by \(O_m(1/K)\).
Now \(H_{2K}-H_K\to\log2\), which gives the second formula.
\end{proof}

Equivalently, \eqref{mass:seed} at \(\alpha=0,\ R=1\) gives \(Z(u)=-[\psi(1+u)+\psi(1-u)-2\psi(1)]/(2u^2)\), with \(Z(0)=\zeta(3)\).
\subsection{Even and odd samples}
At \(u=k\pi/\phi\), the defining binomial series terminate for
\(\phi=\pi/3,\pi/2,\pi\). The scale in the first column clears
the common squared denominator in each sample.
\begin{equation}\label{rm:parity}
\begin{array}{lcl}
 \text{sample} & k & \text{value}\\ \hline
 18k^2 A_{\pi/3}(3k) & \text{even}
   & H_{3k}-2H_{3k/2}-3H_k+2H_{k/2}\\
                     & \text{odd}
   & 3H_{3k}-H_k-4(\Omega_{3k}-\Omega_k)-4/(3k)\\
 18k^2 B_{\pi/3}(3k) & \text{even}
   & -H_{6k}-H_{2k}+H_{3k}-H_{3k/2}-H_k+H_{k/2}\\
                     & \text{odd}
   & H_{6k}+H_{2k}+H_{3k}-H_k-2(\Omega_{3k}-\Omega_k)-2/(3k)\\
 k^2 A_{\pi}(k) & \text{even} & -4\Omega_k\\
                & \text{odd} & 4\Omega_k-2/k\\
 k^2 B_{\pi}(k) & \text{even} & -2H_{2k}-H_k+H_{k/2}\\
                & \text{odd} & 2H_{2k}-H_k+2\Omega_k-1/k\\
 4k^2 A_{\pi/2}(2k) & \text{even} & -2\Omega_k\\
                    & \text{odd} & 2\Omega_k-1/k
\end{array}
\end{equation}
To obtain the entries from the fixed-angle formulas, write
\(\mu=k\pi/\phi\). Near \(\mu\) one has
\(\cos(\phi u)=(-1)^k+O((u-\mu)^2)\). Every scaled \(Z\) has at most a
simple pole. Thus replacing the cosine by \((-1)^k\) leaves the constant
term unchanged. The sine sums for \(A_{\pi/3},B_{\pi/3}\) are
regular at their sampling arguments, since the numerator
\(\sin(2j\pi/3)\) vanishes at multiples of three. Their factors
\(\sin(\phi u)\) therefore make their contributions zero.

For example, let \(\widehat Z(x)\) denote the constant term of \(Z\)
at \(x\), equal to \(Z(x)\) at a regular point. With
\(c=(-1)^k\), the \(A\)-row of \eqref{rm:fixedtable} gives
\[
 A_{\pi/3}(3k)=
 (c/2-1)\widehat Z(3k)+\tfrac14\widehat Z(3k/2)
 +\tfrac{c+2}{18}\widehat Z(k)-\tfrac1{36}\widehat Z(k/2).
\]
Insert Lemma~\ref{rm:poles} according to the parity of \(k\). The same substitution in the \(B\)-row of \eqref{rm:fixedtable} and
the closed forms at \(\pi,\pi/2\) gives the other entries.
Equivalently, they follow directly from the mass formulas using
\(\Omega_{2j}=H_{2j}-H_j/2\) and \eqref{spec:pf}.

The endpoint identities also give a short check on the \(1/n\) weight.
Putting \(c=\cos(\pi u)\) in the formulas for \(N_\pi,A_\pi\) gives
\[
 N_\pi+A_\pi=(1-c)\left(\frac{4\log2}{u^2}-\tfrac12 Z(u/2)\right),
 \qquad A_\pi-N_\pi=(1+c)\left(4Z(u)-\tfrac12 Z(u/2)\right).
\]
The first prefactor has a double zero at even \(k\), the second at odd
\(k\). The brackets have at most simple poles there, yielding
\(N_\pi(k)=(-1)^{k+1}A_\pi(k)\).

\section{Which coefficients cancel at \texorpdfstring{\(\pi/3\)}{pi/3}?}\label{fix:section}
Write \(h=\pi/3\), \(d_n=n^2\binom{2n}{n}\), and omit \(h\) from the subscripts of
\(N,A,B\). By separating even indices and then residues modulo three,
\begin{equation}\label{fix:sines}
 S_h(u)=S_{2h}(u)+\tfrac12 S_{2h}(u/2),\qquad
 S_{2h}(u)=\tfrac{\sqrt3}{2}\sum_{j\ge0}L_{-3}(2j+2)u^{2j}\quad(|u|<1),
\end{equation}
where \(L_{-3}(s)=\sum_{\nu\ge0}[(3\nu+1)^{-s}-(3\nu+2)^{-s}]\).
The cosine terms reduce to ordinary zeta values by \eqref{fix:filters}.
Multiply each generating function by \(a(u)=hu/\sin(hu)\), with \(a(0)=1\).
The coefficient of \((-1)^ju^{2j}\) now has \(e_j(n)\) replaced by \(F_j(n)\), where
\begin{equation}\label{fix:F}
 \frac{x}{\sinh x}=\sum_{k\ge0}\lambda_k x^{2k},\qquad
 F_j(n)=\sum_{r=0}^j\lambda_{j-r}h^{2j-2r}e_r(n).
\end{equation}

\begin{theorem}\label{fix:class}
For \(m\ge0\), consider the sums
\[
 \sum_{j=0}^m\pi^{2m-2j}\sum_{n\ge1}\frac{F_j(n)}{d_n}
       \left(\frac{a_j}{n}+b_jH_n+c_jH_{2n}\right)
 \qquad(a_j,b_j,c_j\in\mathbb Q).
\]
All displayed \(L_{-3}\)-terms cancel precisely when
\begin{equation}\label{fix:planes}
 (1+1/(2\cdot4^j))a_j+b_j/2-4^jc_j=0\qquad(0\le j\le m).
\end{equation}
The solution space has dimension \(2m+2\). A basis is obtained by taking
the coefficient arrays of \(P_j\) and \(Q_j\) at level \(j\), with the
factor \(\pi^{2m-2j}\), where
\begin{equation}\label{fix:PQ}
\begin{aligned}
 P_j&=\sum_{n\ge1}\frac{F_j(n)}{d_n}(H_{2n}+2\cdot4^jH_n),\\
 Q_j&=\sum_{n\ge1}\frac{F_j(n)}{d_n}(H_{2n}-H_n+4^j/n).
\end{aligned}
\end{equation}
Every such sum is a rational combination of \(\pi^{2m-2i}\zeta(2i+3)\), \(0\le i\le m\).
\end{theorem}
\begin{proof}
The sine parts of \(aN,aA,aB\) are \(2h S_h(u),h S_{2h}(u),-2h S_{2h}(2u)\)
by Theorem~\ref{ang:main}. At degree \(2j\), \eqref{fix:sines} gives just
\(L_{-3}(2j+2)\); its multiplier in the stated sum is the left side of
\eqref{fix:planes} times \((-1)^j\sqrt3\,h\pi^{2m-2j}\).
The other terms follow by expanding \(Z(v)=\sum_{i\ge0}\zeta(2i+3)v^{2i}\).
Each equation has two independent solutions \((0,2\cdot4^j,1)\) and \((4^j,-1,1)\)
and involves only its own triple. This proves both assertions.
\end{proof}

\begin{corollary}\label{rm:raw}
Use \(e_r\) in the sum and denote the three coefficients by
\(\alpha_r,\beta_r,\gamma_r\). The term containing \(L_{-3}(2i+2)\)
is \((-1)^i\sqrt3\,\pi^{2m-2i+1}L_{-3}(2i+2)\) times
\begin{equation}\label{rm:q}
 q_i=\sum_{r=i}^m
 \frac{(1+1/(2\cdot4^i))\alpha_r+\beta_r/2-4^i\gamma_r}{3^{2(r-i)+1}(2r-2i+1)!}.
\end{equation}
\end{corollary}
\begin{proof}
Multiply the three normalized series by \(\sin(hu)/(hu)\) and collect degree
\(2r\). Its sine factor gives the factorial in \eqref{rm:q}. Since \(F_r=e_r\)
plus lower terms, the change of basis in \eqref{fix:F} is invertible over
\(\mathbb Q\) after the powers of \(\pi\) have been included.
Thus \(q_i=0\) for every \(i\) gives the same dimension and basis.
\end{proof}
This classifies termwise cancellation. A converse concerning the values of the sums
would require information about rational relations among the constants.

For a single \(e_m\) and one triple \((\alpha,\beta,\gamma)\), \eqref{rm:q} requires
\(-\gamma y^2+(\alpha+\beta/2)y+\alpha/2=0\) for \(y=1,4,\ldots,4^m\).
At \(m=1\) the roots force the ratio \((-8,26,1)\). At \(m\ge2\) they force
the zero triple. The corresponding weight-five identity and a weight-seven
identity with a lower-order term are
\begin{align}
 \sum_{n\ge1}\frac{e_1(n)}{d_n}(H_{2n}+26H_n-8/n)
   &=\frac{34\pi^2\zeta(3)-287\zeta(5)}{54},\label{fix:sun}\\
 \sum_{n\ge1}\frac{e_2(n)-\pi^2e_1(n)/180}{d_n}(H_{2n}+296H_n-128/n)
   &=
   \frac{61\pi^4\zeta(3)}{3645}
   -\frac{20153\pi^2\zeta(5)}{9720}+\frac{24563\zeta(7)}{1296}.
   \label{fix:seven}
\end{align}
Indeed their left sides are \(3P_1-2Q_1+\pi^2(9P_0-8Q_0)/54\) and
\(9P_2-8Q_2+(7\pi^2/540)(33P_1-32Q_1)\), respectively. Insert the
values in \eqref{fix:values}. Equation~\eqref{fix:sun} is the evaluation
proposed by Sun \cite[Conj.~3.10, (3.14)]{Sun2026}. Write its bracket as
\(H_{2n-1}+26H_{n-1}+37/(2n)\). Davydychev and Kalmykov give the
harmonic components and a log-sine form for the last sum
\cite[eqs.~(2.33), (2.82), (2.83)]{DK}. Akhilesh gives a direct evaluation
of the repeated-square sum with the extra \(1/n\) \cite[Thm.~6]{Akhilesh}.

\section{Fixed-angle formulas}\label{fix:tables}
Grouping the cosine terms by their residues gives
\begin{equation}\label{fix:filters}
\begin{aligned}
C_\pi(u)&=\tfrac14 Z(u/2)-Z(u),&
C_{2\pi/3}(u)&=\tfrac1{18}Z(u/3)-\tfrac12Z(u),\\
C_{\pi/2}(u)&=\tfrac1{32}Z(u/4)-\tfrac18 Z(u/2),&
C_{\pi/3}(u)&=\tfrac14 C_{2\pi/3}(u/2)-C_{2\pi/3}(u).
\end{aligned}
\end{equation}
Put \(c=\cos(\phi u)\), \(s=\sin(\phi u)\). At \(\phi=\pi/3\) write
\(M=4A(2u)+2cB(u)\) and \(R=N(2u)+c[B(u)-A(u)]\).
Substitution of \eqref{fix:filters} in Theorems~\ref{ang:main}
and~\ref{ang:relations} gives the coefficients of each row below.
The poles in the full functions are removable.
\begingroup\small\setlength{\arraycolsep}{3pt}
\begin{equation}\label{rm:fixedtable}
\begin{array}{l|rrrrrr}
 &Z(u/6)&Z(u/3)&Z(u/2)&Z(u)&Z(2u/3)&Z(2u)\\ \hline
 N-2sS_{\pi/3}(u)/u&c/36&-c/9&-c/4&c-2&0&0\\
 A-sS_{2\pi/3}(u)/u&-1/36&(c+2)/18&1/4&c/2-1&0&0\\
 B+2sS_{2\pi/3}(2u)/u&-1/72&1/18&1/8&-1/2&2c/9&2c\\
 M&-c/36&(c-1)/9&c/4&1-c&(8c^2+2)/9&8c^2-6\\
 R&c/72&-(2c+1)/36&-c/8&(1+2c-4c^2)/4&1/9&4c^2-3
\end{array}
\end{equation}
\endgroup
With \(a(u)=(\pi u/3)/\sin(\pi u/3)\), the relation \(ca(2u)=a(u)\) gives
\begin{equation}\label{fix:generating}
 \sum_{j\ge0}P_j(-u^2)^j=\tfrac12 a(2u)M(u),\qquad
 \sum_{j\ge0}Q_j(-u^2)^j=a(2u)R(u)\quad (|u|<3/2).
\end{equation}
Taylor expansion of \(Z\) at zero gives the next table. In row \(j\) the
entries list the coefficients of the odd zeta values in increasing order.
The entry for \(\zeta(2i+3)\) carries the factor \(\pi^{2j-2i}\).
\begin{equation}\label{fix:values}
\begin{array}{c|c|c}
 j&P_j&Q_j\\ \hline
 0&5/3&2/3\\
 1&(19/54,-77/18)&(49/162,-203/54)\\
 2&(-187/17496,-4249/2916,62665/3888)&(-961/87480,-4123/2916,20429/1296)
\end{array}
\end{equation}
Both entries at \(j=0\) follow from the first two sums in \cite{Chavan}.
They give \(B(0)+2N(0)/3=\zeta(3)\) and \(P_0=5\zeta(3)/3\), hence
\(Q_0=3\zeta(3)/2-P_0/2\).
Expansion of the first three rows of \eqref{rm:fixedtable} through degree two
gives the coefficients of \(\sqrt3\pi^3L_{-3}(2),\sqrt3\pi L_{-3}(4),
\pi^2\zeta(3),\zeta(5)\) in \(\sum e_1(n)x_n/d_n\), respectively,
\[
 \begin{array}{c|rrrr}
 x_n&\multicolumn{4}{c}{\text{coefficients}}\\ \hline
 1/n&1/108&-3/8&1/27&29/27\\
 H_n&1/324&-1/6&5/162&34/81\\
 H_{2n}&-1/162&4/3&10/81&-1237/162
 \end{array}
\]

For a closed formula at any level let \(b=4^k\) and set
\[
 c_{k,j}=-\frac{36^{-j}}{72}+\frac{(1-b)9^{-j}}{18}
 +\frac{4^{-j}}8+\frac{b-1}2
 +\frac{b(6-b)(4/9)^j}{9(2-b)}+\frac{b(3b-2)4^j}{2-b}.
\]
Then
\begin{equation}\label{co:closed}
 P_m=\sum_{j=0}^m(-1)^j\lambda_{m-j}(\pi/3)^{2m-2j}
        c_{m-j,j}\zeta(2j+3).
\end{equation}
In fact, put \(u=ix,z=\pi x/3,\ A=z/\sinh z,\ A_2=2z/\sinh(2z)\)
and \(D=z\coth z\). In the \(M\)-row of \eqref{rm:fixedtable}
the six coefficients after multiplication by \(a(2ix)/2\) are
\(-A/72,(A-A_2)/18,A/8,(A_2-A)/2,(4D+A_2)/9,4D-3A_2\).
Since \(z\coth z=z\coth(2z)+z/\sinh(2z)\), the coefficient of \(z^{2k}\)
in \(D\) is \(4^k\lambda_k/(2-4^k)\). Multiplying the convergent power
series gives \eqref{co:closed}. The coefficient of \(\zeta(2m+3)\) is
\((-1)^m(72\cdot144^m+40\cdot16^m+9\cdot9^m-1)/(72\cdot36^m)\).
The normalization is unique. If \(g(0)=1\) and \(g(2x)\cosh z=g(x)\) near zero,
iteration gives \(g(x)=g(x/2^j)2^j\sinh(z/2^j)/\sinh z\to z/\sinh z\)
for every \(g\) continuous at zero.

At \(\phi=\pi\), put \(d=(1-c)\log(2)/u^2\). The same cosine reductions give
\[
 \begin{aligned}
 N_\pi&=2d-2(1+c)Z(u)+\tfrac{c}{2}Z(u/2),\\
 A_\pi&=2d+2(1+c)Z(u)-\tfrac12 Z(u/2),\\
 B_\pi&=d+8cZ(2u)+Z(u)-\tfrac14 Z(u/2).
 \end{aligned}
\]
At \(\pi/2\) the sine term in \(A\) vanishes. At \(2\pi/3\) use
\(S_{4\pi/3}=-S_{2\pi/3}\). In each line \(c\) is evaluated at the stated angle.
\begin{equation}\label{fix:simplifications}
 \begin{aligned}
 A_{\pi/2}&=(1-c)\log(2)/u^2+\tfrac{1+c}{4}Z(u/2)-\tfrac1{16}Z(u/4),\\
 2A_{2\pi/3}+N_{2\pi/3}&=(1-c)\bigl[3\log(3)/u^2-\tfrac29Z(u/3)\bigr].
 \end{aligned}
\end{equation}
Extracting \((-u^2)^r\) in the second line, with \(v=2\pi/3\), gives
\[
 \sum_{n\ge1}\frac{3^ne_r(n)}{d_n}(2H_n+1/n)
 =\frac{3v^{2r+2}\log3}{(2r+2)!}
 +2\sum_{k=1}^r\frac{(-1)^{k-1}v^{2r+2-2k}\zeta(2k+1)}{9^k(2r+2-2k)!}.
\]
The coefficients of \(-u^2\) in the other four formulas give
\begin{equation}\label{fix:coefficienttable}
 \begin{array}{ccl}
 T&x_n&\sum_{n\ge1}T^ne_1(n)x_n/d_n\\ \hline
 2&H_n&\pi^4\log(2)/384+\pi^2\zeta(3)/32-31\zeta(5)/256\\
 4&1/n&\pi^4\log(2)/12-3\pi^2\zeta(3)/4+31\zeta(5)/8\\
 4&H_n&\pi^4\log(2)/12+\pi^2\zeta(3)-31\zeta(5)/8\\
 4&H_{2n}&\pi^4\log(2)/24+4\pi^2\zeta(3)-527\zeta(5)/16 .
 \end{array}
\end{equation}

\subsection{The moment identities}\label{app:moments}
At any angle, write \(F_m^\phi(n)=\sum_{r=0}^m\lambda_{m-r}\phi^{2m-2r}e_r(n)\), using \(\lambda_j\) from \eqref{fix:F}. For \(q\ge1\) one obtains
\begin{equation}\label{mom:general}
 \sum_{n\ge1}\frac{t^{2n}F_m^\phi(n)}{n^2\binom{2n}{n}}H_{qn}
    =(-1)^m\sum_{k\ge1}\frac{\rho_k(q)}{\mu_k^{2m+2}} .
\end{equation}
Indeed, expand \eqref{spec:pf} at \(s=0\). For \(|s|<\mu_1^2\) the series of absolute values is bounded by \(\sum_k\rho_k(q)/(\mu_k^2-|s|)\). Expansion of \eqref{intro:normal} gives the left side. Subtraction yields the same formula with \(H_{qn}-cH_{pn}\) and the corresponding difference of masses.

At \(\phi=\pi/3\) put \(p_m=(-1)^mP_m\), with \(P_m\) from \eqref{fix:PQ}. Equations \eqref{mass:third} and \eqref{mom:general} give
\begin{equation}\label{mom:positive}
 p_m=\sum_{k\ge1}\frac{\rho_k^B+2\cdot4^m\rho_k^A}{(9k^2)^{m+1}}.
\end{equation}
In particular, \(p_m\) is the \(m\)th moment of the measure with masses \(\rho_k^B/(9k^2)\) at \(1/(9k^2)\) and \(2\rho_k^A/(9k^2)\) at \(4/(9k^2)\).
Call this measure \(\nu\). Its support is infinite, so the matrices
\((p_{i+j})\) and \((p_{i+j+1})\) are Gram matrices of independent monomials
for \(\nu\) and \(y\,d\nu\), respectively. Its two largest atoms, read from
\eqref{mass:third}, have masses \(11/27\) at \(4/9\) and \(637/1080\) at \(1/9\).
The remaining support is contained in \([0,4/81]\). Thus
\[
 p_m=\frac{11}{27}(4/9)^m+\frac{637}{1080}(1/9)^m+O((4/81)^m),
 \qquad p_{m+1}/p_m\to4/9.
\]
For \(V(s)=T_2(s)+2T_1(4s)\) the same measure gives
\(V(s)=\int(1+sy)^{-1}d\nu(y)\), so \(V(s)\to0\) by dominated convergence.
Differentiation under the integral gives the complete monotonicity of
\(V\). For the related hyperbolic sine quotient see \cite[Lemma~2]{YT}.
The finite geometric identity gives the exact remainder after degree
\(d\), namely \((-s)^{d+1}\int y^{d+1}(1+sy)^{-1}d\nu(y)\).
The same argument applies to each \(T_q\) or \(T_q-cT_p\) in
Theorem~\ref{spec:main}, using its mass divided by \(\mu_k^2\) at \(1/\mu_k^2\).

For example, \eqref{mom:positive} with \(m=0\) gives
\(\sum_{k\ge1}(\rho_k^B+2\rho_k^A)/k^2=15\zeta(3)\).
Separately, \eqref{ang:A}, \eqref{ang:B} and \eqref{fix:sines} at \(u=0\) give
\[
 \sum_{k\ge1}\frac{\rho_k^A}{k^2}=
   \frac{3\sqrt3\,\pi}{2}L_{-3}(2)-\zeta(3),\qquad
 \sum_{k\ge1}\frac{\rho_k^B}{k^2}=
   17\zeta(3)-3\sqrt3\,\pi L_{-3}(2).
\]
At \(\phi=\pi\), the value \(A_\pi(0)=\pi^2\log2+7\zeta(3)/2\), together with the endpoint mass for \(A\), gives
\(\sum_{k\ge1}\Omega_k/k^2=\pi^2\log(2)/8+7\zeta(3)/8\).

\section{A common factor selecting two sine orders}\label{co:section}
Keep \(\phi=\pi/3\), put \(\alpha=\pi/6\), and write
\(P(\theta)=\sum_{j=1}^m p_j\theta^{2j-1}\).
For \(j\ge1\), the coefficient of \(u^{2j-2}\) in \eqref{ang:seed} gives
\(\sum_{n\ge1}e_{j-1}(n)t^{2n}/d_n=4^j\arcsin(t/2)^{2j}/(2j)!\)
for \(0\le t\le1\). Set \(w_P(n)=\sum_{j=1}^m p_j(2j-1)!\,e_{j-1}(n)/4^j\).
For constants \(A,D\) its harmonic sum has the integral
\begin{align}
 I(P)&=\sum_{n\ge1}\frac{w_P(n)}{d_n}\left(H_{2n}+AH_n-\frac{D}{2n}\right)
       =\int_0^\alpha P(\theta)K(\theta)\,d\theta,\label{co:insertion}\\
 K(\theta)&=-(A+1)\log(1-2\sin\theta)-A\log(1+2\sin\theta)
                    +D\log(2\sin\theta).\notag
\end{align}
Indeed the generating function \(G(t)=\sum_n w_P(n)t^{2n}/d_n\) satisfies
\(G'(t)\,dt=P(\theta)\,d\theta\) for \(t=2\sin\theta\).
The integrals of \(G'\) against \(-\log(1-t),-\log(1-t^2),-2\log t\)
insert \(H_{2n},H_n,1/n\) term by term. This follows by integration by parts
or by the power series of the logarithm. Each fixed \(e_j(n)\) is bounded
by \(\zeta(2)^j/j!\), so \(\sum_n |w_P(n)|(H_{2n}+H_n+1)/d_n\)
converges. This justifies the three interchanges.

\begin{lemma}[Fourier moment]\label{co:fourierlemma}
Put \(T_s=\sum_{k\ge1}\sin(2\pi k/3)/k^s\), \(p=4^{-r}\), \(q=9^{-r}\),
\(z_r=2p\), and
\[
 R_r=1+Ap-\tfrac12(1-q)[1+(A+D)p-Dp^2],\qquad
 V_r=p[D-\tfrac12(2A+1)(1-p)(1-q)].
\]
Then
\begin{align}
 I(P)={}&\sum_{r=1}^m(-1)^{r-1}
  [P^{(2r-1)}(\alpha)R_r+P^{(2r-1)}(0)V_r]\zeta(2r+1)\notag\\
 &+\sum_{r=1}^m(-1)^rP^{(2r-2)}(\alpha)[1+(D-A)z_r+Dz_r^2]T_{2r}.
 \label{co:momentformula}
\end{align}
\end{lemma}
\begin{proof}
The factorization \(1\pm2\sin\theta=4\sin((\alpha\pm\theta)/2)
\cos((\alpha\mp\theta)/2)\) turns the log-sine Fourier expansions of \(K\) into
\[
 2(2A+1)\sum_{k\ {\rm even}}\frac{\cos(k\alpha)\cos(k\theta)}k
 +2\sum_{k\ {\rm odd}}\frac{\sin(k\alpha)\sin(k\theta)}k
 -D\sum_{k\ge1}\frac{\cos(2k\theta)}k .
\]
All indices are positive. Repeated integration by parts gives both trigonometric
moments as the real and imaginary parts of
\[
 \int_0^\alpha P(\theta)e^{ik\theta}\,d\theta
   =\sum_{j=0}^{2m-1}\frac{(-1)^j}{(ik)^{j+1}}
          [e^{ik\alpha}P^{(j)}(\alpha)-P^{(j)}(0)] .
\]
Here for complex \(P\) one applies the calculation to its real and imaginary
parts separately. Let \(C_s(v)=\sum_{k\ge1}\cos(kv)/k^s\) just in this proof.
At order \(s=2r+1\), the upper and lower odd derivatives give, after removing
\((-1)^{r-1}\),
\[
 (1+Ap)\zeta(s)+(A+1)p C_s(4\alpha)-(1+Dp)C_s(2\alpha),
 \qquad p[D\zeta(s)-(2A+1)C_s(2\alpha)].
\]
Grouping residues gives \(C_s(4\alpha)=(q-1)\zeta(s)/2\),
\(C_s(2\alpha)=(1-p)(1-q)\zeta(s)/2\). The even derivative at \(0\)
vanishes. At \(\alpha\) its multiplier, with the same sign removed, is
\((A+1)z_r T_{2r}-(1+Dz_r)(1+z_r)T_{2r}\) by \eqref{fix:sines}.
These are precisely the three coefficients in \eqref{co:momentformula}.
For the Fourier expansion insert an Abel factor \(\rho^k\) in each original
log-sine series. For \(1/2<\rho<1\), its sum \(\log|1-\rho e^{iv}|\)
is bounded in absolute value on this interval of integration by a constant
times \(1+|\log\theta|+|\log(\alpha-\theta)|\).
Dominated convergence removes \(\rho\) in the integrals. After integration
the Fourier terms are \(O(k^{-2})\), so their series has the same limit.
\end{proof}

Use the Lidstone polynomials \cite{Waldschmidt}, defined by
\(\ell_0(t)=t,\ \ell_j''=\ell_{j-1},\ \ell_j(0)=\ell_j(1)=0\) for \(j\ge1\).
Integration twice shows that \(\ell_j\) is odd of degree \(2j+1\), with rational
coefficients, leading coefficient \(1/(2j+1)!\), and
\(\ell_j^{(2i)}(1)=\delta_{ij}\).
Since its second derivative has constant sign on \((0,1)\), induction gives
\(\operatorname{sgn}\ell_j=(-1)^j\) there and
\(\ell_j'(0)=-\int_0^1(1-t)\ell_{j-1}(t)\,dt\ne0\) for \(j\ge1\).
The series \(\sinh(yt)/\sinh t=\sum_{j\ge0}\ell_j(y)t^{2j}\), \(|t|<\pi\),
follows from the same recurrence and endpoint conditions. Hence
\(\ell_j'(0)=\lambda_j\) from \eqref{fix:F} and
\(\ell_j(y)=\sum_{i=0}^j\lambda_{j-i}y^{2i+1}/(2i+1)!\).
In particular \(\lambda_0=1\) and \(\lambda_j=-\sum_{i=1}^j\lambda_{j-i}/(2i+1)!\).

\begin{theorem}[Two selected orders]\label{co:pair}
For integers \(1\le r<s\) put
\begin{align}
 D&=2^{2r+2s-2},\qquad A=D+2^{2r-1}+2^{2s-1},\label{co:weights}\\
 P(\theta)&=\alpha^{2s-1}(2s-1)!
   [\ell_{s-1}(\theta/\alpha)-(\lambda_{s-1}/\lambda_{r-1})\ell_{r-1}(\theta/\alpha)].
   \label{co:Ptwomodes}
\end{align}
This is the unique monic odd polynomial with \(P'(0)=0\) and only the
endpoint even derivatives of orders \(2r-2,2s-2\) nonzero. Write
\begin{equation}\label{co:connection}
 E_{r,s}(n)=\frac{4^s w_P(n)}{(2s-1)!}
          =F_{s-1}(n)-(\pi/3)^{2s-2r}\frac{\lambda_{s-1}}{\lambda_{r-1}}F_{r-1}(n).
\end{equation}
Then
\begin{align}
 &\sum_{n\ge1}\frac{E_{r,s}(n)}{d_n}\left(H_{2n}+AH_n-\frac D{2n}\right)\notag\\
 &\qquad=\frac{4^s}{(2s-1)!}\sum_{j=1}^s(-1)^{j-1}
  [P^{(2j-1)}(\alpha)R_j+P^{(2j-1)}(0)V_j]\zeta(2j+1).\label{co:identity}
\end{align}
Conversely, every nonzero odd polynomial with \(P'(0)=0\) for which the
coefficient of each \(T_{2j}\) in \eqref{co:momentformula} vanishes has
this form up to a constant,
for a unique pair \(r,s\) and the corresponding \(A,D\).
\end{theorem}
\begin{proof}
Expand \(P(\alpha t)=\sum b_j\ell_j(t)\). Then
\(b_j=\alpha^{2j}P^{(2j)}(\alpha)\) and \(\sum b_j\lambda_j=0\).
The nonzero \(b_j\) must correspond to roots \(z_{j+1}\) of
\(1+(D-A)z+Dz^2\). There can be at most two and the slope condition
excludes just one. If the roots are \(z_r,z_s\), factorization as
\((1-2^{2r-1}z)(1-2^{2s-1}z)\) gives \eqref{co:weights}; the slope
condition and the monic leading term give \eqref{co:Ptwomodes}.
Conversely these choices cancel both terms in the Fourier moment.
Expanding the two \(\ell_j\)'s by their finite sums gives \eqref{co:connection}
using \(4\alpha^2=(\pi/3)^2\). Equation~\eqref{co:identity} now follows
from \eqref{co:insertion} and \eqref{co:momentformula}.
\end{proof}
The pairs \((1,2),(2,3)\) give \eqref{fix:sun} and \eqref{fix:seven}.
For \((3,4)\) the polynomial is \(\theta^7-18\alpha^2\theta^5/7+11\alpha^4\theta^3/7\).
Substituting it in \eqref{co:identity} gives the first evaluation below.
Taking \(m=3\) in \eqref{co:closed} gives the second. We suppress \(n\)
in the three \(e_j(n)\).
\begin{align}
 &\sum_{n\ge1}\frac{e_3-\pi^2e_2/147+11\pi^4e_1/476280}{d_n}
                      (H_{2n}+4256H_n-2048/n)\notag\\
 &\qquad=-\frac{1502\pi^6\zeta(3)}{4822335}
          -\frac{56377\pi^4\zeta(5)}{1224720}
          +\frac{4515127\pi^2\zeta(7)}{571536}
          -\frac{2569391\zeta(9)}{34992},\label{co:orderexample}\\
 &\sum_{n\ge1}\frac{e_3-\pi^2e_2/54+7\pi^4e_1/29160-31\pi^6/11022480}{d_n}
                      (H_{2n}+128H_n)\notag\\
 &\qquad=\frac{15007\pi^6\zeta(3)}{33067440}
          +\frac{76463\pi^4\zeta(5)}{1574640}
          +\frac{1238647\pi^2\zeta(7)}{209952}
          -\frac{2241263\zeta(9)}{34992}.\label{co:secondnine}
\end{align}

\appendix

\section{Logarithmic integrals and Parseval's formula}\label{app:integrals}

Put \(b=\phi/2\), \(s_0=\sin b\). In \eqref{ang:int}, \(X=N_\phi,A_\phi,B_\phi\)
has weight \(x_n=1/n,H_n,H_{2n}\), respectively. Uniform expansion of
\(\sin(2u\theta)/u\) on \(0\leq\theta\leq b\), for \(u\) in any compact set,
and integrability of \(K_X\) give
\begin{equation}\label{app:logmom}
 \sum_{n\ge1}\frac{t^{2n}e_r(n)x_n}{n^2\binom{2n}{n}}
 =\frac{2^{2r+2}}{(2r+1)!}\int_0^b\theta^{2r+1}K_X(\theta)\,d\theta
 \qquad(r\ge0).
\end{equation}
\begin{proposition}\label{app:logints}
For \(0<b\leq\pi/2\), write \(x=\sin\theta/s_0\). Then
\begin{align}
 \int_0^b\theta\log\frac{1+x}{x^2}\,d\theta
   &=\frac{\zeta(3)-\Cl_3(2b)}4+\frac{b^2}{2}\log(2s_0),\label{app:logone}\\
 \int_0^b\big[(\theta^3+3b^2\theta)\log(1+x)-8\theta^3\log x\big]\,d\theta
   &=\frac{21}{8}\big[\Cl_5(2b)-\zeta(5)\big]\notag\\
   &\quad+\frac{3b^2}{4}\big[\Cl_3(2b)+6\zeta(3)\big]
         +\frac{b^4}{4}\log(2s_0).\label{app:logthree}
\end{align}
\end{proposition}
\begin{proof}
In the first relation of Theorem~\ref{ang:relations}, take the constant
coefficient divided by four and the coefficient of \(-u^2\) multiplied by
\(3/8\). Formula \eqref{app:logmom}, with \(K_N=-2\log x\) and
\(K_B-K_A=\log(1+x)\), gives exactly the two integrals.
\end{proof}
Taking \(b=\pi/6\) in \eqref{app:logone} gives
\begin{equation}\label{main:int}
 \int_0^{\pi/6}\theta\log\frac{1+2\sin\theta}{4\sin^2\theta}\,d\theta
     =\frac{\zeta(3)}6 .
\end{equation}
 The same
formula at \(b=\pi/3,\pi/2\) and \eqref{app:logthree} at \(b=\pi/6\) give
\begin{align*}
 \int_0^{\pi/3}\theta\log\frac{3+2\sqrt3\sin\theta}{4\sin^2\theta}\,d\theta
 &=\frac{13\zeta(3)+\pi^2\log3}{36},\\
 \int_0^{\pi/2}\theta\log\frac{1+\sin\theta}{\sin^2\theta}\,d\theta
 &=\frac{7\zeta(3)}{16}+\frac{\pi^2\log2}{8},\\
 \int_0^{\pi/6}\!\left[\left(\theta^3+\frac{\pi^2\theta}{12}\right)
       \log(1+2\sin\theta)-8\theta^3\log(2\sin\theta)\right]d\theta
 &=\frac{19\pi^2\zeta(3)-203\zeta(5)}{144}.
\end{align*}
Grouping residues in the cosine series gives
\(\Cl_3(2\pi/3)=-4\zeta(3)/9\), \(\Cl_3(\pi)=-3\zeta(3)/4\),
\(\Cl_3(\pi/3)=\zeta(3)/3\), and \(\Cl_5(\pi/3)=25\zeta(5)/54\).
Applying \eqref{app:logmom} with \(r=1\) to Sun's identity gives
\[
 \int_0^{\pi/6}\theta^3\log\frac{(2\sin\theta)^{16}}
  {(1-2\sin\theta)^{27}(1+2\sin\theta)^{26}}\,d\theta
 =\frac{34\pi^2\zeta(3)-287\zeta(5)}{144}.
\]

\begin{proposition}[Parseval's formula]\label{app:parseval}
Put \(J_\phi(u)=(1-\cos(\phi u))/u^2\), with its limiting value at zero,
and \(K_J=1\). For \(X,Y\in\{N_\phi,A_\phi,B_\phi,J_\phi\}\),
\[
 \int_0^\infty u^2 X(u)Y(u)\,du
 =\pi\int_0^b K_X(\theta)K_Y(\theta)\,d\theta .
\]
\end{proposition}
\begin{proof}
Extend the kernels by zero past \(b\). Their logarithmic endpoint
singularities are square integrable. In the unitary sine transform
normalization, \eqref{ang:int} gives
\(uX(u)=\sqrt{2\pi}\,\widehat K_X(2u)\), also for \(X=J_\phi\) by direct
integration. Parseval with \(v=2u\) gives the identity, and Cauchy--Schwarz
gives absolute convergence \cite[Ch.~III]{Titchmarsh}.
\end{proof}
At \(\phi=\pi/3\), the classical integral
\(\int_0^{\pi/6}\log^2(2\sin\theta)\,d\theta=7\pi^3/216\)
\cite{BS,Lewin} gives
\(\int_0^\infty u^2N_{\pi/3}(u)^2\,du=7\pi^4/54\).
At \(\phi=\pi\), \(K_N=-2\log\sin\theta\), \(K_A=-2\log\cos\theta\). Hence
\begin{align*}
 \int_0^\infty u^2N_\pi(u)^2\,du
 =\int_0^\infty u^2A_\pi(u)^2\,du
 &=\frac{\pi^4}{6}+2\pi^2\log^2 2,&
 \int_0^\infty u^2(A_\pi-N_\pi)^2\,du&=\frac{\pi^4}{2}.
\end{align*}
Indeed, \(-\log(2\sin\theta)\) and \(-\log\tan\theta\) have zero mean
and cosine coefficients \(1/k\) and \((1-(-1)^k)/k\) at frequency \(2k\).
Orthogonality gives their squared norms \(\pi^3/24\) and \(\pi^3/8\).
Theorem~\ref{ang:main} identifies the bracket below with \((A_\pi-N_\pi)/4\),
including the removable values at odd integers. Thus
\[
 \int_0^\infty u^2
 \left[(1+\cos\pi u)\sum_{\substack{k\ge1\\k\ {\rm odd}}}
        \frac{1}{k(k^2-u^2)}\right]^2du=\frac{\pi^4}{32}.
\]
Pairing \(J_\phi\) with \(N_\pi\) and with \(3A_{\pi/3}-N_{\pi/3}\) gives
\begin{align*}
 \int_0^\infty (1-\cos\pi u)N_\pi(u)\,du&=\pi^2\log2,\\
 \int_0^\infty (1-\cos(\pi u/3))(3A_{\pi/3}(u)-N_{\pi/3}(u))\,du&=0 .
\end{align*}
On \((0,\pi/6)\) the kernel integrals are \(\int K_N=\Cl_2(\pi/3)\) and
\(\int K_A=\Cl_2(2\pi/3)/2\). The latter equals \(\Cl_2(\pi/3)/3\).

\section{The odd arcsine expansion}\label{app:odd}

The sine expansion gives a companion family. Put
\(\Omega_{2n+1}=\sum_{j=0}^n(2j+1)^{-1}\) and
\[
 v_n(u)=\frac{\binom{2n}{n}}{4^n(2n+1)}
       \prod_{j=1}^n\left(1-\frac{4u^2}{(2j-1)^2}\right).
\]
\begin{lemma}\label{app:oddexpansion}
For \(0\le\theta\le\pi/2\) and \(u\in\mathbb C\),
\(\sum_{n\ge0}v_n(u)\sin^{2n+1}\theta=\sin(2u\theta)/(2u)\),
with value \(\theta\) at \(u=0\).
\end{lemma}
\begin{proof}
The function \(y(x)=\sin(2u\arcsin x)/(2u)\) solves
\((1-x^2)y''-xy'+4u^2y=0\), \(y(0)=0,\ y'(0)=1\).
The coefficient of \(x^{2n+1}\) starts with \(v_0=1\) and obeys
\((2n+3)(2n+2)v_{n+1}=((2n+1)^2-4u^2)v_n\), giving the stated product.
On \(|u|\le R\) the product is bounded by
\(\prod_{j\ge1}(1+4R^2/(2j-1)^2)=\cosh(\pi R)\).
Stirling's formula gives \(v_n(u)=O_R((n+1)^{-3/2})\).
Uniform convergence extends the identity to \(x=\sin\theta=1\) and \(u=0\).
\end{proof}
\begin{proposition}\label{app:oddclosed}
With \(b=\phi/2,\ s_0=\sin b,\ c=\cos(u\phi),\ s=\sin(u\phi)\), one has
\begin{align}
 \sum_{n\ge0}\frac{v_n(u)s_0^{2n+1}}{2n+1}
 &=\frac{s}{2u}\log(2s_0)+\frac c2S_\phi(u)-\frac{us}{2}C_\phi(u),\label{app:oddfirst}\\
 \sum_{n\ge0}2\Omega_{2n+1}v_n(u)s_0^{2n+1}
 &=\frac{s}{u}\log(4s_0)+\frac c2S_{2\phi}(u)
       -\frac{us}{2}[Z(u)+C_{2\phi}(u)].\label{app:oddsecond}
\end{align}
The sums are entire in \(u\) and give the removable values on the right.
\end{proposition}
\begin{proof}
Write \(g(\theta)=\sin(2u\theta)/(2u)\), \(\theta_y=\arcsin(s_0y)\).
The two sums integrate \(g(\theta_y)/y\) and
\(2[g(b)-yg(\theta_y)]/(1-y^2)\) over \(0<y<1\), since
\(\int_0^1y^{2n}dy=1/(2n+1)\) and
\(\int_0^12(1-y^{2n+2})/(1-y^2)dy=2\Omega_{2n+1}\).
Termwise integration follows from the estimate in the preceding proof
and \(\Omega_{2n+1}=O(\log(n+2))\).
In the first integral \(dy/y=d\log(2\sin\theta)\).
In the second, separate \(2g(b)/(1+y)\) and use
\(dK_A=2y\,dy/(1-y^2)\). Integration by parts gives, respectively,
\[
 g(b)\log(2s_0)-\int_0^b g'\log(2\sin\theta)\,d\theta,\qquad
 2g(b)\log2+\int_0^b g'K_A\,d\theta.
\]
The boundary terms at the singular ends vanish since
\(g(\theta)=O(\theta)\) at zero and \(g(b)-g(\theta)=O(b-\theta)\)
while \(K_A=O(1+|\log(b-\theta)|)\) at \(b\).
Use the \(L^2\) cosine series of the kernels in the proof of
Theorem~\ref{ang:main}. For \(u\notin\mathbb Z\) their pairing with
\(g'(\theta)=\cos(2u\theta)\) uses
\[
 \int_0^b g'(\theta)\cos(2k\theta)\,d\theta
   =\frac{kc\sin(k\phi)-us\cos(k\phi)}{2(k^2-u^2)} .
\]
The double-angle formulas give \eqref{app:oddfirst}--\eqref{app:oddsecond}.
Local uniform convergence of the defining sums fills in the integers.
\end{proof}
For each row of the table the middle entry is \(\sum_{n\ge0}v_n(u)h_n\)
and the last is its value at zero, where \(v_n(0)=4^{-n}\binom{2n}{n}/(2n+1)\).
Write \(\Omega=\Omega_{2n+1}\) in the first column.
\[
\begin{array}{c|c|c}
 h_n & \displaystyle\sum_{n\ge0}v_n(u)h_n & u=0\\ \hline
 \Omega &
   \dfrac{\sin(\pi u)}{2u}\bigl[2\log2-u^2 Z(u)\bigr] &\pi\log2\\[4pt]
 2^{-n}\Omega &
   \dfrac{\sin(\pi u/2)}{2\sqrt2\,u}
       \bigl[3\log2-\tfrac{u^2}{4}Z(u/2)\bigr] &\dfrac{3\sqrt2\,\pi\log2}{8}\\[4pt]
 (3/4)^n(2\Omega+\tfrac{1}{2n+1}) &
   \dfrac{\sin(2\pi u/3)}{2\sqrt3\,u}
       \bigl[\log432-\tfrac{2u^2}{9} Z(u/3)\bigr]&\dfrac{\pi\log432}{3\sqrt3}
\end{array}
\]
For the first two rows, \(S_{2\phi}=0\) at \(\phi=\pi,\pi/2\) and
\(Z(u)+C_\pi(u)=Z(u/2)/4\). For the third, at \(\phi=2\pi/3\) the two
sine terms cancel and \(Z(u)+2C_{2\pi/3}(u)=Z(u/3)/9\).
Both cosine identities follow by grouping residue classes.
Finally, at \(\phi=\pi/3,\ u=0\), subtract \(4/3\) times \eqref{app:oddfirst}
from twice \eqref{app:oddsecond}. Since
\(\Cl_2(2\pi/3)=2\Cl_2(\pi/3)/3\), this gives
\[
 \sum_{n\ge0}\frac{\binom{2n}{n}}{16^n(2n+1)}
       \left(2\Omega_{2n+1}-\frac{2}{3(2n+1)}\right)=\frac{2\pi}{3}\log2.
\]

\section{A classical proof at \texorpdfstring{$\pi/3$}{pi/3}}\label{oh:classical}
Here \(N=N_{\pi/3}\), \(X_1=A_{\pi/3}\), \(X_2=B_{\pi/3}\),
\(c(u)=\cos(\pi u/3)\) and \(g(u)=\sin(\pi u/3)/u\).
Write \(F(a,b;d)={}_2F_1(a,b;d;1/4)\).

\begin{theorem}[Classical reduction]\label{oh:main}
Let \(M(u)\) be the explicit sum of six \(Z\)-values in the \(M\)-row
of \eqref{rm:fixedtable}. Then
\begin{equation}\label{oh:master}
 4X_1(2u)+2c(u)X_2(u)=M(u).
\end{equation}
Bailey's cubic transformation, the Gauss connection formulas and
Clausen's square identity give an independent proof of this formula
and of both coefficient families in \eqref{fix:PQ}.
\end{theorem}

\subsection{Three values from Bailey's transformation}
Parameter derivatives insert the harmonic weights \cite{DK,KWY}. More precisely, for
\(\mathcal A={}_3F_2(1+u,1-u,1;a,b;1/4)\), evaluated at \(a=3/2,b=2\),
\[
 X_1=\tfrac12(1-\partial_b)\mathcal A,\qquad
 X_2=\tfrac14(3-\partial_a-\partial_b)\mathcal A.
\]
Indeed the term of index \(n-1\) in \(\mathcal A/2\) is \(w_n(u)\).
The two derivatives in \(a,b\) multiply it by
\(-(2H_{2n}-H_n-2)\) and \(-(H_n-1)\), respectively.
All parameter derivatives below are at the argument \(1/4\).
On compact parameter sets away from lower-parameter poles the series and
their fixed derivatives converge normally by the ratio test and Cauchy's
formula. We calculate first at generic parameters near the stated values.

Put \(\Phi={}_3F_2(p,q,1-q;A,B;z)\),
\(A=(p+q+1)/2,\ B=(p-q+2)/2,\ a=p/3,\ \alpha_j=a+j/3\).
Bailey's transformation is
\[
 \Phi=(1-4z)^{-p}\,{}_3F_2\bigl(\alpha_0,\alpha_1,\alpha_2;
                 A,B;-27z/(1-4z)^3\bigr).
\]
Use \(\delta_0=1,\ \delta_1=1/3,\ \delta_2=2/3\), and put
\[
 R_j=\frac{4^{\alpha_j}\Gamma(A)\Gamma(B)\Gamma(a+\delta_j)}
 {\Gamma(p+1)\Gamma(A-\alpha_j)\Gamma(B-\alpha_j)}.
\]
At \(z=1/4\), with \(\vartheta=z\,d/dz\), the values are
\begin{equation}\label{oh:theta}
 \Phi=R_0,\qquad \vartheta\Phi=a(R_1-R_0),\qquad
 \vartheta^2\Phi=a^2R_0-a(2a-1/3)R_1+aR_2 .
\end{equation}
Put \(t=1-4z\). The connection coefficient for \(\alpha_j\) at infinity is
\(\Gamma(A)\Gamma(B)\prod_{k\ne j}\Gamma(\alpha_k-\alpha_j)/
 [\Gamma(A-\alpha_j)\Gamma(B-\alpha_j)\prod_{k\ne j}\Gamma(\alpha_k)]\).
The differences of exponents are nonintegral. Approaching \(t=0\) from
above fixes all powers. Gamma triplication gives the three branches
\(b_jt^j(1-t)^{-\alpha_j}(1+O(t^3))\), with
\((b_0,b_1,b_2)=(R_0,-aR_1,aR_2/2)\). The last factor is a convergent
series in \(-4t^3/[27(1-t)]\). Two derivatives with
\(\vartheta=-(1-t)d/dt\) give \eqref{oh:theta}, then meromorphic continuation
allows exceptional parameters whenever the limits exist. The cubic and
the first two values appear in \cite[(4.05)]{Bailey} and
\cite[(9), (22)--(24)]{KE}. The infinity formula is
\cite[(16.8.7)--(16.8.8)]{DLMF}.

\subsection{The two parameter derivatives}
Write
\begin{align*}
 V(u)&=\sum_{n\ge1}w_n(u)\sum_{k=1}^{n-1}\frac{k}{k^2-u^2},\qquad
 K=V-2X_2+X_1,\\
 D(u)&=\left.\frac{d}{dt}F(\tfrac12+t+u,\tfrac12+t-u;\tfrac32+2t)\right|_{t=0}.
\end{align*}
The apparent poles of \(V\) cancel term by term against the corresponding
product factors in \(w_n\). The resulting products are locally uniformly
bounded in \(u\), and the inner sum contributes at most a constant times
\(H_{n-1}\). Thus \(V\) is entire.
Write \(\Phi_\pm\) for shifting the first upper parameter alone to
\(p\pm1\). Coefficient comparison gives
\(\Phi_+=\Phi+\vartheta\Phi/p=(2R_0+R_1)/3\) at \(1/4\). For
\(d=(A-p)(B-p)\), the lower shift satisfies
\[
 d\Phi_-=\{(1-z)\vartheta^2+(1/2-z)\vartheta+d-zq(1-q)\}\Phi .
\]
This follows by multiplying each coefficient of \(\Phi_-\) by \(p+n-1\)
to give \((p-1)\) times the coefficient of \(\Phi\) and using the
coefficient recurrence of \(\Phi\). At \(1/4\) this gives \(d\Phi_-=T\), where
\[
 T=\tfrac16(p-1)(2p-3)R_0-\tfrac16p(p-1)R_1+\tfrac14pR_2 .
\]
Throughout, a \(p\)-derivative keeps \(q\) fixed. On \(p=q+1\),
\(\Phi_-=2T_p/(q-1/2)\) for \(q\ne1/2\). At \(p=3/2,q=1/2\) its
value is \(2T_{pp}\), since \(d=(p-3/2)^2/4\) there at fixed \(q\).

In the next two sums \(q\) takes the values \(1/2\pm u\). With
\(z=1/4\) understood, termwise differentiation gives
\begin{align}
 K(u)&=-\frac{\sqrt3}{2u^2}\sum_q(\Phi_+)_p(q-1,q),\notag\\
 D(u)&=2\sum_q\Phi_p(q+1,q)+\frac{g(u)}{1/4-u^2}
                  -\sum_q\Phi_-(q+1,q)/q .\label{oh:derivatives}
\end{align}
For the first equality, twice the derivative sum is
\(\partial_t F(1/2+t+u,1/2+t-u;1/2+2t)|_0\).
Euler transformation gives \((2/\sqrt3)F(t+u,t-u;1/2+2t)\).
Its Gauss series before the factor \(2/\sqrt3\) has term
\(-u^2w_n(u)\) at \(t=0\), multiplied upon differentiation by
\(2\sum_{k<n} k/(k^2-u^2)-2(2H_{2n}-H_n)\), for \(n\ge1\).
For the second equality the paired derivatives of the cancelled
upper and lower parameters involve
\(\sum_{k=0}^{n-1}1/(q+1+k)=\sum_{k=0}^{n-1}1/(q+k)+1/(q+n)-1/q\).
The end terms sum to the last two terms of \eqref{oh:derivatives},
using \(F(1/2+u,1/2-u;3/2)=g(u)\).

Clausen's square formula \cite[(16.12.2)]{DLMF}, differentiated with
upper parameters first \(t\pm u\), then \(1/2+t\pm u\), gives
\begin{align}
 V(u)-X_2(u)+N(u)/2&=c(u/2)K(u/2)/2,\notag\\
 N(2u)&=c(u)K(u)/2+g(u)^2-g(u)D(u)/4 .\label{oh:cl}
\end{align}
In the first calculation the factor \((2t)_n/(4t)_n\) for \(n\ge1\)
has value \(1/2\) and logarithmic derivative \(-2H_{n-1}\) at zero
after cancellation. For the second put \(m=n+1\). The undifferentiated
term is \(2w_m(2u)\) and its logarithmic derivative is
\(4\sum_{k<m}k/(k^2-4u^2)-4H_{2m}-2/m+8\).
Summing gives \eqref{oh:cl} since \(F(u,-u;1/2)=c(u)\)
and \(2\sum_m w_m(2u)=g(u)^2\). The trigonometric Gauss values follow
by substituting \(z=\sin^2 v\) in the Gauss equation and taking \(v=\pi/6\).

Let \(L(u)=\sum_{n\ge1}\chi_{-3}(n)/(n^2-u^2)\), where
\(\chi_{-3}\) is \(1,-1,0\) on the classes \(1,2,0\bmod3\), and put
\[
 U=4Z(2u)+Z(u)+\tfrac49 Z(2u/3)-\tfrac19 Z(u/3),
 \qquad W=2L(2u)+L(u).
\]
The gamma values in \eqref{oh:derivatives} reduce to
\begin{equation}\label{oh:KD}
 K=-cU/2+\sqrt3\,gW/2,\qquad
 D=4g+u^2g\{16Z(2u)-U\}-\sqrt3\,cW .
\end{equation}
Here \(c,g\) take argument \(u\). We give the reduction rules to make
both derivative calculations reproducible. Write \(\psi=\Gamma'/\Gamma\).
Logarithmic differentiation of \(R_j\) gives the single rule
\begin{equation}\label{oh:logrule}
 \frac{(R_j)_p}{R_j}
 =\frac{\log4+\psi(a+\delta_j)}3+\frac{\psi(A)+\psi(B)}2-\psi(p+1)
                -\frac{\psi(A-\alpha_j)+\psi(B-\alpha_j)}6 .
\end{equation}
On \(p=q-1\), reflection and duplication give
\((R_0,R_1)=(\sin(\pi q/3+\pi/6),\,2\cos(\pi q/3))\).
On \(p=q+1,\ q=1/2+v\), the three values are
\[
 g(v),\quad [vg(v)+\sqrt3 c(v)]/p,\quad
          2q[\sqrt3c(v)-vg(v)]/(3p).
\]
For detail on reducing the derivatives, define
\(\Psi_r^\pm(v)=\psi(r+v)\pm\psi(r-v)\). Grouping the partial
fractions and using duplication gives
\begin{align*}
 \Psi_1^+(v)&=-2\gamma-2v^2Z(v),&
 \Psi_{1/2}^+(v)&=2\Psi_1^+(2v)-\Psi_1^+(v)-4\log2,\\
 \Psi_{1/3}^-(v)-\Psi_{2/3}^-(v)&=18vL(3v),&
 \Psi_{5/6}^-(v)-\Psi_{1/6}^-(v)&=-72vL(6v)-18vL(3v).
\end{align*}
For \(K\), write \(d_j(v)=(R_j)_p/R_j\) on \(p=q-1,\ q=1/2+v\).
Substitution of \eqref{oh:logrule}, shifting by \(\psi(x+1)-\psi(x)=1/x\)
and using gamma triplication gives
\(\sum_{v=\pm u}(d_0(v)+d_1(v))=u^2U\) and
\((d_0-d_1)(u)-(d_0-d_1)(-u)=-3uW\).
For \(D\), let \(d_j(v)\) denote the same derivatives on \(p=q+1\).
With \(a_0=1/4-u^2,\ a_1=9/4-u^2\), the same rules give
\begin{align*}
 \sum_{v=\pm u}(4d_0+d_1+d_2)(v)
   &=12+3u^2(16Z(2u)-U)-2/a_0-6/a_1,\\
 (d_1-d_2)(u)-(d_1-d_2)(-u)&=2u/a_0-3uW .
\end{align*}
In \eqref{oh:derivatives} these two lines have multipliers
\(g/3,\sqrt3c/(3u)\). Differentiating the polynomial factors of \(T\)
and adding \(g/a_0\) leaves \(2g/(3a_0)+2g/a_1-2\sqrt3c/(3a_0)\).
This cancels the rational terms and proves \eqref{oh:KD}.
An alternative check of the four \(d_j\) sums uses one multiplication
rule throughout. Express \(Z,L\) by the first and third \(\Psi\) rules.
For each remaining \(\psi(r+su)\) with \(s\ne0\) take \(m=3|s|\)
and replace it by
\(\log m+m^{-1}\sum_{h=0}^{m-1}\psi((r+h)/m+\operatorname{sgn}(s)u/3)\).
Here every such \(m\) is a positive integer. Shift the constant
arguments into \((0,1]\) and collect the equal digammas. This leaves
only rational terms and the constants \(\psi(1),\psi(1/2),\psi(3/2)\),
whose values follow from \(\psi(1)=-\gamma\), duplication and shift.
The second line of \eqref{oh:cl} now gives the \(N\)-row of
\eqref{rm:fixedtable}, with
\(S_{\pi/3}(u)=\sqrt3[L(u)+L(u/2)/2]/2\).

\subsection{The remaining Gauss connection}
For \(F_\ell(t,v)=F(t+v,t-v;\ell)\), the two required quadratic
transformations followed by connection about \(1\) are
\begin{align*}
 F_{t+1/2}(t,v)&=a_t F_{1/2}(t/2,v/2)
          +\tfrac12 b_t F_{3/2}((1+t)/2,v/2),\\
 F_{3/2+t}(1/2,v)&=A_t F_{1/2}(-t/2,v/2)
          +\tfrac12 B_t F_{3/2}((1-t)/2,v/2).
\end{align*}
Writing \(\Gamma(x\pm y)=\Gamma(x+y)\Gamma(x-y)\), their coefficients are
\[
 a_t=\frac{\Gamma(t+1/2)\Gamma(1/2)}{\Gamma((1+t)/2\pm v/2)},\quad
 b_t=\frac{\Gamma(t+1/2)\Gamma(-1/2)}{\Gamma(t/2\pm v/2)},\quad
 A_t=\frac{(2t+1)b_t}{v^2-t^2},\quad B_t=-(2t+1)a_t .
\]
Indeed the quadratic maps send \(1/4\) to \(3/4\).
In each connection the second solution has factor \((1/4)^{1/2}=1/2\).
For the second line Euler's transformation of the connected functions
cancels the initial factor \((3/4)^{1/2+t}\).
All bases are positive \cite[(15.8.18), (15.8.20), (15.10.21)]{DLMF}.

Set \(P(u)=(\partial_1F_{1/2})(0,u)=-2u^2V(u)\) and
\(R(u)=u\,(\partial_1 F_{3/2})(1/2,u)\), where \(\partial_1\)
differentiates the common upper-parameter shift. Put \(C=\cos(\pi u), S=\sin(\pi u)\).
At \(v=2u,t=0\) the four coefficients are \(C,2uS,S/(2u),-C\).
The terms from differentiating those coefficients are
\(G=a'_0c+b'_0g/2,\ I=A'_0c+B'_0g/2\) at \(v=2u\).
Logarithmic differentiation and duplication as above give
\begin{align*}
 G(u)&=u^2Cc[8Z(2u)-Z(u)]+uSg[u^2Z(u)-2\log2],\\
 I(u)&=Sc[1-\log2+u^2Z(u)/2]/u-Cg[1+u^2(8Z(2u)-Z(u))/2].
\end{align*}
The two differentiated connections therefore read
\[
 \binom{P(2u)}{R(2u)}
 =\begin{pmatrix}C&S\\S&-C\end{pmatrix}\binom{P(u)}{R(u)}
 +\binom{8u^2K(2u)+2G(u)}{2u[D(2u)-2I(u)]}.
\]
For the first left side before rearranging, differentiating both uppers
with lower slope one gives \(P(v)/2-v^2K(v)\). For the second the
lower derivative alone is \([D(v)-R(v)/v]/2\).
Apply the matrix identity at \(u\) and \(2u\), and eliminate \(R(u)\).
For \(H(u)=4V(2u)+2c(u)V(u)\) the result is
\begin{align}
 2H(2u)-[c(2u)-1/2]H(u)={}&-8K(4u)-4C^2K(2u)-C^2G(u)/u^2\notag\\
 &-G(2u)/(2u^2)-SC[D(2u)-2I(u)]/u .\label{oh:hrec}
\end{align}
Here \(2c(u)[c(2u)-1/2]=C\). An even analytic solution near zero is
unique, since the first nonzero coefficient of degree \(2m\) in a
difference would be multiplied by \(2\cdot4^m-1/2\).

From \eqref{oh:cl} and the definition of \(K\),
\(X_2=V-c(u/2)K(u/2)/2+N/2\) and
\(X_1=V+K-c(u/2)K(u/2)+N\).
Thus the left side of \eqref{oh:master} is \(H+E\), with
\[
 E(u)=4K(2u)-4c(u)K(u)+4N(2u)-c(u)c(u/2)K(u/2)+c(u)N(u).
\]
It suffices to insert \(H=M-E\) in \eqref{oh:hrec}. The following
finite coefficient calculation gives this check along with the second
one needed below.

\subsection{The second generating function}
Write \(b_m=2\cdot4^m,\ J_m=(P_m-Q_m)/b_m,\ h=c(u/2)\), with
\(P_m,Q_m\) as in \eqref{fix:PQ}, and \(a(u)=(\pi u/3)/\sin(\pi u/3)\).
The two generating functions near zero are
\begin{align}
 \sum_{m\ge0}(-1)^mP_m u^{2m}&=\frac{a(u)}{2c(u)}M(u),&
 \sum_{m\ge0}(-1)^mJ_m u^{2m}&=a(u)\mathcal B(u),\label{oh:both}
\end{align}
where
\[
 \mathcal B(u)=\frac{c(u)+1}{9}Z(u/3)+\frac{2c(u)+3-4h}{8}Z(u/2)
      +\frac{4h-1}{72}Z(u/6)+\frac h8 Z(u/4)-\frac h{72}Z(u/12).
\]
The first identity follows from \eqref{oh:master} by taking coefficients.
Since \(a(u/2)=h a(u)\), the definition of \(J_m\) gives the second with
\(\mathcal B=X_1-N/2+hX_1(u/2)/2\). The expressions for \(X_1,X_2\) above and
\eqref{oh:master} at \(u/2\) reduce the required check to
\begin{equation}\label{oh:Bcheck}
 \mathcal B(u)=M(u/2)/4+hK(u/2)/2-hc(u/4)K(u/4)/4-N(u)/2+hN(u/2)/4 .
\end{equation}

For clarity, here are the polynomial operations for both finite checks.
Form left minus right in \eqref{oh:hrec} with \(H=M-E\), and in
\eqref{oh:Bcheck}. Substitute \(G,I,K,D\) from above and the \(M,N\)
rows of \eqref{rm:fixedtable}. Regard each \(Z(ru),L(ru)\) as a symbol.
Replace every \(\cos(k\pi u/12),\sin(k\pi u/12)\) by
\((y^k+y^{-k})/2,(y^k-y^{-k})/(2i)\), respectively.
The following are the scales of the resulting coefficient lists. To
reproduce each list, set all symbols to zero for the constant, then
turn on one symbol with value one and subtract the constant. Expand
the numerator in \(y\) after clearing denominators.
\[
 \begin{array}{c|l|l|c}
  & Z\text{ scales} & L\text{ scales} &\text{nonzero numerators, including constant}\\ \hline
 \eqref{oh:hrec}&1/6,1/3,1/2,2/3,1,4/3,2,4&1/2,1,2,4&0\\
 \eqref{oh:Bcheck}&1/12,1/6,1/4,1/3,1/2,1&1/4,1/2,1&0
 \end{array}
\]
For instance, the \(Z(u/6)\) coefficient in the first check factors as
\(c(u)[2c(2u)-1][-2c(u/2)^2+c(u)+1]/72=0\).
This polynomial calculation proves both substitutions.
Uniqueness in \eqref{oh:hrec} now proves \eqref{oh:master}.
Continuation extends it to the entire left side. The series in \eqref{oh:both} may
be expanded near zero by normal convergence.

\begin{proposition}\label{oh:pairprovenance}
Every scalar sum in Theorem~\ref{co:pair} is a finite combination of
coefficients in \eqref{oh:both}.
\end{proposition}
\begin{proof}
For \(r<s\), put \(d=b_{r-1}b_{s-1}\) and
\(\kappa=(\pi/3)^{2s-2r}\lambda_{s-1}/\lambda_{r-1}\).
By \eqref{co:connection} the common factor is \(F_{s-1}-\kappa F_{r-1}\).
At either index \(j=r-1,s-1\), the summand for \(P_j+dJ_j\)
has weights \(1,\ b_j+d(1+1/b_j),\ -d/2\) on \(H_{2n},H_n,1/n\).
These are precisely the weights in \eqref{co:weights}. Thus the scalar is
\((P_{s-1}+dJ_{s-1})-\kappa(P_{r-1}+dJ_{r-1})\).
For example \eqref{fix:seven} equals
\(P_2+256J_2+(7\pi^2/540)(P_1+256J_1)\).
From \eqref{oh:both} the first three companion coefficients are
\[
 J_0=\tfrac12\zeta(3),\quad J_1=\tfrac{\pi^2}{162}\zeta(3)-\tfrac7{108}\zeta(5),
 \quad J_2=\tfrac{13\pi^4}{1399680}\zeta(3)-\tfrac{7\pi^2}{5184}\zeta(5)
                          +\tfrac{689}{62208}\zeta(7).
 \]
Together with \(Q_m=P_m-b_mJ_m\), this evaluates both basis
families by classical transformations.
\end{proof}

\end{document}